\documentclass[12pt]{article}   
   
\usepackage{latexsym,amsfonts,amsmath,epsfig,tabularx,amsthm,amssymb,enumitem,bm}
\usepackage{float}

\usepackage[title,titletoc]{appendix}

\usepackage{microtype}
\usepackage{dutchcal}

\usepackage[draft=false,colorlinks,bookmarksnumbered,linkcolor=black, citecolor=black]{hyperref}

\usepackage{aliascnt}

\newtheorem{theorem}{Theorem}[section]

\newaliascnt{lem}{theorem}
\newtheorem{lemma}[lem]{Lemma}

\aliascntresetthe{lem}

\newaliascnt{ass}{theorem}

\aliascntresetthe{ass}

\newaliascnt{prop}{theorem}
\newtheorem{prop}[prop]{Proposition}

\aliascntresetthe{prop}

\newaliascnt{cor}{theorem}
\newtheorem{corollary}[cor]{Corollary}

\aliascntresetthe{cor}

\newaliascnt{defi}{theorem}
\newtheorem{defi}[defi]{Definition}

\aliascntresetthe{defi}

\theoremstyle{definition}
\newaliascnt{ex}{theorem}

\aliascntresetthe{ex}

\newaliascnt{rem}{theorem}
\newtheorem{remark}[rem]{Remark}

\aliascntresetthe{rem}

\renewcommand{\d}{\,\mathrm{d}}											
\renewcommand*{\epsilon}{\varepsilon}                                   
\renewcommand*{\rho}{\varrho}                                   		
\newcommand*{\sep}{\; \vrule \;}                                        
\newcommand*{\N}{\mathbb{N}}                                            
\newcommand*{\R}{\mathbb{R}}                                            
\newcommand*{\C}{\mathbb{C}}                                            
\newcommand*{\Z}{\mathbb{Z}}                                            
\newcommand*{\B}{\mathcal{B}}                                           
\renewcommand*{\S}{\mathcal{S}}                                         
\newcommand*{\D}{\mathcal{D}}                                         	
\newcommand*{\A}{\mathcal{A}}                                         	
\renewcommand{\d}{\,\mathrm{d}}						
\newcommand*{\ba}{\bm{a}}						
\newcommand*{\bc}{\bm{c}}						
\newcommand*{\be}{\bm{e}}						

\newcommand*{\x}{\bm{x}}						
\newcommand*{\bj}{\bm{j}}                        
\newcommand*{\bk}{\bm{k}}                        

\newcommand*{\abs}[1]{\left| #1 \right|}                                
\newcommand*{\norm}[1]{\left\| #1 \right\|}                             
\newcommand*{\floor}[1]{\left\lfloor #1 \right\rfloor}                  
\newcommand*{\ceil}[1]{\left\lceil #1 \right\rceil}                     

\DeclareMathOperator{\supp}{supp}										

\DeclareMathOperator{\Id}{Id}	                                        
\DeclareMathOperator{\id}{id}                                           

\usepackage{footmisc}

\usepackage{tikz}

\newcommand{\hookdownarrow}{\mathrel{\rotatebox[origin=c]{90}{$\hookleftarrow$}}}
\usetikzlibrary{arrows,decorations.pathmorphing,decorations.pathreplacing,backgrounds,positioning,fit,petri}
\usepackage[margin=0.6cm, format=hang]{caption}

\setlist{itemsep=2pt, topsep=2pt}

\title{Rate-optimal sparse approximation of\\ compact break-of-scale embeddings}   

\author{
Glenn Byrenheid\footnote{Friedrich-Schiller-University Jena, Institute of Mathematics, Ernst-Abbe-Platz 2, 07737 Jena. Email: \href{mailto:glenn.byrenheid@uni-jena.de}{glenn.byrenheid@uni-jena.de}}
\qquad
Janina Hübner\footnote{\emph{Corresponding author}. Ruhr University Bochum, Faculty of Mathematics, Research Group Numerics, Universit\"atsstra{\ss}e 150, 44801 Bochum. Email: \href{mailto:janina.huebner@rub.de}{janina.huebner@rub.de}.}
\qquad	
Markus Weimar\footnote{Ruhr University Bochum, Faculty of Mathematics, Research Group Numerics, Universit\"atsstra{\ss}e 150,
44801 Bochum, Germany. Email: \href{mailto:markus.weimar@rub.de}{markus.weimar@rub.de}}
} 
\begin{document}   
\maketitle

\begin{abstract}
\noindent
The paper is concerned with the sparse approximation of functions having hybrid regularity borrowed from the theory of solutions to electronic Schr\"odinger equations due to Yserentant~\cite{YS10}. We use hyperbolic wavelets to introduce corresponding new spaces of Besov- and Triebel-Lizorkin-type to particularly cover the energy norm approximation of functions with dominating mixed smoothness. Explicit (non-)adaptive algorithms are derived that yield sharp dimension-independent rates of convergence. 


\end{abstract} 

\section{Motivation and main result}

The electronic Schr\"odinger equation describes the motion of a huge system of electrons under Coulomb interaction forces in a field of clamped nuclei. 
It forms the basis of modern quantum chemistry. 
Solutions to this equation, so-called wave functions, depend on $d=3N$ variables (three spatial dimensions for each of the $N\gg 1$ electrons) and thus are hard to approximate numerically in general.
Nonetheless, in a series of articles Yserentant and co-authors proved that physically relevant solutions (those which respect the so-called Pauli principle) possess a special 
type of smoothness that connects classical (isotropic) Sobolev regularity with square integrable mixed weak derivatives of order up to $N+1$; see~\cite{YS04,YS10}. 
As we will show, this kind of \emph{hybrid smoothness} can help to reduce the numerical effort of the high-dimensional problem at hand drastically.
This is the initial motivation for us to study the approximation problem in spaces $H^{r,s}_{p,q}X$, where $X\in\{B,F\}$, of multivariate functions with hybrid regularity of Besov- or Triebel-Lizorkin-type which particularly cover standard $L_p$-Sobolev spaces $H_p^s$ and $S_p^r H$ of isotropic and dominating mixed smoothness~\cite{Tri2019}, respectively, as special cases.
The central question considered in this paper is the optimal worst-case (non-)linear approximability of functions w.r.t.\ the norms in $H^{r,s}_{p,q}F$ or $H^{r,s}_{p,q}B$, respectively. 
Our interest is founded by the analysis of Galerkin discretizations of elliptic partial differential equations (PDEs). 
In this context, Céa's lemma allows to bound the norm of the resulting error in the respective \emph{energy space} $H$ by the best approximation error w.r.t.\ the underlying Galerkin subspace.
In the simplest case of the Dirichlet problem for Poisson's equation on a bounded domain $\Omega\subset\R^d$, we have $H=H^1_0(\Omega)$ and hence 
$$
  \norm{u -u_m \sep H^1(\Omega)} \lesssim \inf_{v\in \mathbb{V}_m} \norm{u-v \sep H^1(\Omega)},
$$
where $u_m\in \mathbb{V}_m \subset H^1_0(\Omega)$ denotes the Galerkin approximation to the solution $u$ with $m$ degrees of freedom.
Taking the supremum over all $u$ in the unit ball of a corresponding function space and the infimum over all linear subspaces $\mathbb{V}_m$ with $\dim(\mathbb{V}_m)\leq m$, yields the so-called Kolmogorov $m$-width~$d_m$; see \cite[Chapter~11]{Pie1980}. 
The rate of convergence of this quantity, as $m$ tends to infinity, is governed by the regularity of the function class under consideration as well as by the target norm~$\norm{\cdot\sep H^1(\Omega)}$. 
In case of Hilbert target spaces, $d_m$ serves as a benchmark for the performance of optimal \emph{linear} algorithms~\cite[Proposition~11.6.2]{Pie1980}.
An appropriate class of source spaces is given by Sobolev-Hilbert spaces~$H^2_{\mathrm{mix}}(\Omega)$ with bounded mixed derivatives up to second order (which coincides with our spaces $H^{2,0}_{2,2}B(\Omega)$ and $H^{2,0}_{2,2}F(\Omega)$, see \autoref{sect:Defis} below). 
In this context, a first result based on hierarchical bases in combination with the introduction of so-called \emph{energy sparse grids} implies that
$$
    d_m \big(\Id\colon H^2_{\mathrm{mix}}((0,1)^d)\to H^1((0,1)^d)\big) \sim m^{-1},
$$ 
see \cite{BuGr04,GrKn00} for details.
Note that, in particular, there is no $d$-dependent logarithmic term as it is known for approximation w.r.t.\ $L_2((0,1)^d)$ or more generally $L_p((0,1)^d)$. A wide overview of this classical $L_p$-situation including 
$$
    d_m \big(\Id\colon H^2_{\mathrm{mix}}((0,1)^d)\to L_2((0,1)^d)\big) \sim (m^{-1}\log^{d-1}m)^2
$$
is provided in the survey \cite{DTU18} and the references therein. 
In connection with measuring errors in the energy norm let us further mention~\cite{ByrDunSic+16} and \cite{Di16}, where the problem of energy norm-based sampling recovery is considered. Additionally, we want to mention \cite{DiUl12}, where $d_m$ is considered in the periodic Hilbert case with special interest to the $d$-dependence of the corresponding constants.

In the realm of PDEs or integral equations on non-smooth domains or manifolds, solutions typically contain singularities caused by irregular points of the underlying geometry~\cite{Gri85}. 
In order to resolve these singular parts numerically, usually iterative schemes based on \emph{adaptive refinement} strategies are employed~\cite{CDD01, DahHarUtz+18}. 
That is, the next Galerkin subspace is chosen during the run time of the algorithm, \emph{depending on} the concrete (unknown) solution~$u$ of interest, rather than being fixed in advance as for linear schemes based on uniform refinement.
Therefore, the rate of convergence of \emph{non-linear} quantities such as best $m$-term widths~$\sigma_m$ (cf.~\autoref{def:QoI}) yields a much better benchmark for such adaptive methods than~$d_m$ discussed above. Again these rates are closely related to the regularity of the underlying function spaces~\cite{DeV1998}.
While the smoothness of solutions with singular parts is known to be quite limited in the scale of Sobolev-Hilbert spaces \cite{CioWei2020}, regularity theory shows that such functions admit higher order smoothness when derivatives are measured w.r.t.\ Lebesgue-norms weaker than $L_2(\Omega)$; see, e.g., \cite{CDK+, DahDeV1997, DahDieHar+16, DahSch21, DahWei15, Han15, HarWei18}.
This finally leads to the observation that best $m$-term widths decay faster than corresponding linear quantities, as this additional regularity can be exploited by adaptive algorithms, but not by linear ones. 
However, in the case of isotropic source and target spaces, usually the optimal rate of convergence is given by the difference in smoothness divided by the dimension~$d$ which is commonly referred to as the \emph{curse of dimensionality}~\cite{Wei15}. 
If both spaces solely possess dominating mixed smoothness, this dimensional dependence in the main rate can be avoided, but still additional $d$-dependent logarithmic factors appear. 
For details and typical results we refer to \cite{DahNovSic2006, HanSic2011, HanSic2012}.
Anyhow, except of \cite{KaePotVol15,KolLomTik21,NguNgu21} which consider the periodic setting and \cite{DauSte10,Nit06,SieWei15} dealing with Hilbert target spaces only, to the best of our knowledge, not much is known for general \emph{break-of-scale embeddings} (such as, e.g., the setting in \autoref{thm:mainsoboloev} below) and/or hybrid-type smoothness spaces.

Since wavelets are known to be a powerful tool in signal processing and numerical analysis~\cite{Dah1997,DeVKun2009}, in this paper we shall focus on algorithms based on a system of \emph{hyperbolic wavelets}.
In contrast to classical isotropic wavelets, their tensor product structure is perfectly suited to resolve anisotropies which naturally arise in various applications, e.g., in physics, engineering, or medical image processing; see \cite{SchUllVed21} and the references therein. 
On the other hand, these wavelets can be employed to characterize function spaces measuring dominating mixed smoothness~\cite{Vyb04} as well as spaces of isotropic regularity~\cite{SchUllVed21}. 
In order to ensure a fair comparison of the performance of linear and non-linear methods, we restrict the corresponding widths $d_m$ and $\sigma_m$ to a \emph{dictionary} $\Psi$ consisting of such hyperbolic wavelets; see \autoref{def:QoI} below for details.

The most important special cases of our \emph{break-of-scale} main result (see \autoref{thm:MainRes}) read as follows:
\begin{theorem}\label{thm:mainsoboloev}
    For $d\in\N$ let $\Omega\subset\R^d$ be a bounded domain, $0< p_0 \leq\infty$, and $1<p_1<\infty$ as well as $r,s\in\R$ such that 
    \begin{align}\label{eq:emb_cond}
        r-\left( \frac{1}{p_0}-\frac{1}{p_1} \right)_+>s>0.
    \end{align}
    \begin{enumerate}[label=(\roman*.)]
        \item If $1<p_0<\infty$, then the embedding $\Id_1 \colon S^{r}_{p_0}H(\Omega) \to H_{p_1}^s(\Omega)$ is compact and there holds
        $$
            d_m\! \left(\Id_1; \Psi\right) \sim m^{-[r-s - (1/p_0-1/p_1)_+]}
            \qquad\text{as well as}\qquad
            \sigma_m\! \left(\Id_1; \Psi\right) \sim m^{-(r-s)}.
        $$

        \item If further $0<q_0\leq\infty$, then the embedding $\Id_2 \colon S^{r}_{p_0,q_0}B(\Omega) \to H_{p_1}^s(\Omega)$ is compact, where
        $$
            d_m\! \left(\Id_2; \Psi\right) \sim m^{-[r-s - (1/p_0-1/p_1)_+]}
            \qquad\text{and}\qquad
            \sigma_m\! \left(\Id_2; \Psi\right) \sim m^{-(r-s)}.
        $$
    \end{enumerate}
\end{theorem}
This theorem reveals several important effects simultaneously. First of all, our convergence rates for the energy norm neither contain a perturbating $d$-dependent logarithm, nor a dimensionally deteriorating main rate. Second, similar to the $L_p$-setting studied in \cite{HanSic2011,HanSic2012},
best $m$-term approximation is not affected by different integrabilities between source and target space. 
While the latter observation resembles a typical feature of non-linear approximation methods, the first one heavily relies on the \emph{break-of-scale structure} of the embeddings under consideration which is expressed by condition~\eqref{eq:emb_cond}: We give up dominating mixed smoothness~$r$ and gain isotropic regularity~$s$.
\medskip

The paper is organized as follows. In \autoref{sect:Defis} we introduce our new function spaces of hybrid smoothness based on hyperbolic wavelets via a characterization by suitably chosen sequence spaces. There we also collect basic properties and recall the definition of the considered approximation widths.
Afterwards, in \autoref{sec:seqapprox} we derive sharp asymptotic approximation rates at the level of sequence spaces. For the upper bounds explicit (non-) linear algorithms are constructed.
Finally, \autoref{sect:fkt_rates} contains the main result on \emph{break-of-scale} embeddings of hybrid-type function spaces.

\textbf{Notation}: By $\N_0$ we denote the set of integers $n\in\Z$ that are larger than or equal to zero and $\R$ denotes the real numbers. 
For $x\in\R$ we further write $x_+:=\max\{x,0\}$. Given two quasi-Banach spaces $X$ and $Y$, we write $X\hookrightarrow Y$ if they are continously embedded, i.e., $X\subseteq Y$ and $\Id\in\mathcal{L}(X,Y)$. We will write $A\lesssim B$ if there exists a constant $c> 0$, such that $A\leq c\cdot B$. 
With $A\sim B$ we mean that $A\lesssim B\lesssim A$. Further, $\abs{D}$ denotes the cardinality of a discrete set $D$.
For $d\in\N$, we use $\S'(\R^d)$ to denote the space of tempered distributions, the topological dual of the Schwartz space $\S(\R^d)$ of rapidly decreasing functions, and $\D'(\Omega)$ is the dual of the space $\D(\Omega)=C^\infty_0(\Omega)$ of test functions with compact support in some open set $\Omega\subset\R^d$.
Finally, the restriction of $g\in \S'(\R^d)$ to $\Omega$ is given by $g\vert_\Omega\in\D'(\Omega)$, where $(g\vert_\Omega)(\varphi):=g(\varphi)$ for all $\varphi\in\D(\Omega)$.

\section{Preliminaries}\label{sect:Defis}
We start with collecting all basic requirements needed later on. 
To do so, we first give a brief introduction to hyperbolic wavelets and define our hybrid function spaces through suitably chosen sequence spaces.
In \autoref{prop:conntoclassspaces} we shall see that in this way we cover several well-known function spaces of interest as special cases.
Afterwards, we formally introduce the widths measuring the performance of optimal (non-)adaptive algorithms and show how their behaviour at the level of function spaces can be reduced to the much simpler sequence space setting.

\subsection{Hyperbolic wavelets and function spaces of hybrid smoothness}
Let us recap some basics about hyperbolic wavelets as described in some more detail in \cite[Section~4]{SchUllVed21}.
Let~$\phi$ be a univariate scaling function and $\psi$ the corresponding wavelet which fulfill the following conditions for some $K\in\N_0$:
\begin{enumerate}[label=(\roman*.)]
    \item $\phi,\psi\in C^K(\R)$ with compact support,
    \item $\norm{\phi\sep L_2(\R)}=\norm{\psi\sep L_2(\R)}=1$, and
    \item $\psi$ has at least $K$ vanishing moments, i.e.
    $$
    \int_{\R}\psi(x)\, x^b\d x=0,\qquad b=0,\ldots, K-1. \quad\text{(In case of $K=0$, this condition is void.)}
    $$
\end{enumerate}
Based on this we define the univariate wavelets via 
$$
    \psi_{j,k}:=2^{-1/2}\,\psi(2^{j-1}\cdot-k)
    \qquad\text{and}\qquad 
    \psi_{0,k}:=\phi(\cdot-k),\qquad
    j\in\N,\, k\in\Z,
$$
such that every $f\in L_2(\R)$ has the wavelet expansion 
$$
    \sum_{j\in\N_0}\sum_{k\in\Z}2^j \left<f,\psi_{j,k}\right>_{L_2}\psi_{j,k}.
$$
In particular, all required properties are fulfilled by the classical Daubechies wavelets.

For the multivariate case, we let $\bj=(j_1,\ldots,j_d)\in\N_{0}^d$ as well as $\bk=(k_1,\ldots,k_d)\in\Z^d$, $d\in\N$ and apply the usual tensor product ansatz to obtain the hyperbolic wavelet functions 
$$
    \psi^{\bj,\bk}(\x)
    := \psi_{j_1,k_1}(x_1)\cdot \ldots \cdot \psi_{j_d,k_d}(x_d),\qquad \x=(x_1,\ldots,x_d)\in\R^d,
$$
that form a basis in $L_2(\R^d)$. 
Moreover, let $\chi$ denote the characteristic function of $[0,1]$ and 
$$
    \chi_{j_i,k_i}:=\chi(2^{j_i} \cdot -k_i),\qquad i=1.\ldots,d,
$$
the characteristic functions of the dyadic intervals $I_{j_i,k_i}:=[2^{-j_i}k_i,2^{-j_i}(k_i+1)]$.
Finally, let
$$
    I^{\bj,\bk}:=I_{j_1,k_1}\times \ldots\times I_{j_d,k_d}
    \qquad\text{and}\qquad \chi^{\bj,\bk}(\x) := \chi_{j_1,k_1}(x_1)\cdot\ldots\cdot\chi_{j_d,k_d}(x_d).
$$
Then $\supp(\psi^{\bj,\bk})\subset c\, \supp(\chi^{\bj,\bk}) = c\, I^{\bj,\bk}$ with some $c>0$ independent of $\bj$ and $\bk$.

For several decades it is well-known that wavelets can be used to describe smoothness and approximation properties of functions and, more general, distributions.
Usually, the point of departure is a fourier-analytic definition of a class of function spaces such as, e.g., the classical Besov or Triebel-Lizorkin spaces $B^s_{p,q}(\R^d)$ and $F^s_{p,q}(\R^d)$ of isotropic smoothness~$s$, respectively, which contains familiar scales like Bessel potential Sobolev spaces $H^s_p(\R^d)$ and H\"older-Zygmund spaces $C^s(\R^d)$ as special cases. We refer to \cite{T06} for a detailed discussion. 
Then a wavelet representation of this class is derived which characterizes the membership of a function in those scales in terms of decay properties of its wavelet coefficients (typically described in terms of sequence spaces).
For hyperbolic wavelets and Besov/Triebel-Lizorkin spaces $S^r_{p,q}X(\R^d)$ (with $X\in\{B,F\}$) of dominating mixed smoothness~$r$, this has been done in~\cite{Vyb04}.
Quite recently, it was found in \cite{SchUllVed21} that \emph{exactly the same} wavelets can be used to characterize also other anisotropic spaces $\widetilde{X}^s_{p,q}(\R^d)$ of Besov- and Triebel-Lizorkin-type which in some cases coincide with the classical (isotropic!) spaces $B^s_{p,q}(\R^d)$ and $F^s_{p,q}(\R^d)$, respectively. 
This is surprising, as previously only isotropic wavelets were employed to describe isotropic spaces.

The structural similarity of the representations of $S^r_{p,q}X(\R^d)$ and $\widetilde{X}^s_{p,q}(\R^d)$ in terms of hyperbolic wavelets inspires the following \emph{wavelet-based definition} of  Besov and Triebel-Lizorkin spaces of hybrid smoothness.

\begin{defi}\label{defi:spaces}
    For $d\in\N$ let $X\in\{B,F\}$, $0<p,q\leq \infty$ (with $p<\infty$ if $X=F$), and $r,s\in\R$. Further let $\{\psi^{\bj,\bk} \sep \bj\in\N_0^d,\, \bk\in\Z^d \}$ be a hyperbolic wavelet system as described above, where $K\in\N_0$ is chosen sufficiently large.
    \begin{enumerate}[label=(\roman*.)]
    \item $H_{p,q}^{r,s}X(\R^d)$ denotes the set of all $f\in\S'(\R^d)$ such that
    $$
        f = \sum_{(\bj,\bk)\in\N_{0}^d\times\Z^d} a_{\bj,\bk}\,\psi^{\bj,\bk} \qquad (\text{convergence in $\S'(\R^d)$})
    $$
    with (unique) coefficients in
    $$
        h_{p,q}^{r,s}x
        := \big\{\ba=(a_{\bj,\bk})_{\bj\in\N_0^d,\bk\in\Z^d} \subset \C \sep \norm{\ba\sep h_{p,q}^{r,s}x}<\infty\big\},
    $$
    where
    \begin{align*}
        \norm{f\sep H_{p,q}^{r,s}X(\R^d)}
        &:=\norm{\ba\sep h_{p,q}^{r,s}x}\\
        &:= \begin{cases}
            \displaystyle\left[ \sum_{\bj\in\N_0^d} 2^{q\left((r-1/p)\abs{\bj}_1 + s\abs{\bj}_{\infty}\right)} \left( \sum_{\bk\in \Z^d} \abs{a_{\bj,\bk}}^{p} \right)^{q/p} \right]^{1/q},&\; x=b,\\[0.8cm]
            \displaystyle\norm{\left(\sum_{\bj\in\N_0^d}2^{q\left(r\,\abs{\bj}_1+s\,\abs{\bj}_{\infty}\right)}\abs{\sum_{\bk\in\Z^d}a_{\bj,\bk}\,\chi^{\bj,\bk}(\cdot)}^{q}\right)^{1/q}\sep L_p(\R^d)},& \; x=f
        \end{cases}
    \end{align*}
(usual modification if $\max\{p,q\}=\infty$). 

    \item Let $\Omega\subset\R^d$ be open. We then define $H_{p,q}^{r,s}X(\Omega)$ via restrictions, i.e.\
    $$
        H_{p,q}^{r,s}X(\Omega)
        := \big\{f\in \D'(\Omega) \sep f=g\vert_{\Omega} \text{ for some } g\in H_{p,q}^{r,s}X(\R^d) \big\}
    $$
    where 
    $$\norm{f\sep H_{p,q}^{r,s}X(\Omega)}:=\inf_{\substack{g\in H_{p,q}^{r,s}X(\R^d),\\f=g\vert_{\Omega}}} \norm{g\sep H_{p,q}^{r,s}X(\R^d)}.
    $$
\end{enumerate}
\end{defi}

\begin{remark}
    Some comments are in order:
    \begin{enumerate}[label=(\roman*.)]
        \item As usual, $p$ indicates the integrability and $q$ is a fine index. 
        Moreover, we shall see that, roughly speaking, $r$ describes the minimal degree of dominating mixed smoothness, while $s$ measures the minimal isotropic regularity of the functions under consideration.
    
        \item Standard arguments show that the introduced spaces are complete w.r.t.\ the given quasi-norms.
        
        \item Since we are only interested in approximation properties and algorithms based on a given, \emph{fixed} system of hyperbolic wavelets, we follow the route taken in \cite{DahWei15}, avoid the usual fourier-analytic detour, and take the expected outcome of a wavelet characterization as a definition.
        The drawback of this approach is that the spaces $H_{p,q}^{r,s}X(\R^d)$ \emph{formally} depend on the concrete choice of the underlying hyperbolic wavelet system. 
        However, \autoref{prop:conntoclassspaces} below and the analysis in \cite{Wei16} indicate that systems with similar properties most likely will lead to the same spaces (up to equivalent quasi-norms). 
        To keep this paper as short as possible, we leave this point as well as the Littlewood–Paley analysis of $H_{p,q}^{r,s}X(\R^d)$ for further research.
    \end{enumerate}
\end{remark}

The following important special cases can be identified. 
Therein, $H^s_p(\R^d)$ and $S^r_p H(\R^d)$ denote the classical $L_p$-Bessel potential spaces (isotropic Sobolev spaces) and $L_p$-Sobolev spaces of dominating mixed smoothness, respectively; see, e.g.,~\cite[(1.17) and (1.41)]{Tri2019}.

\begin{prop}\label{prop:conntoclassspaces} Let $d\in\N$ and $0<p_0,p_1,q_0,q_1\leq \infty$ (with $p_0,p_1<\infty$ for Triebel-Lizorkin spaces), as well as $r_0,r_1,s_0,s_1\in\R$. Then we have
\begin{enumerate}[label=(\roman*.)]
    \item $H^{r_0,s_0}_{p_0,q_0}F(\R^d)=F_{p_1,q_1}^{s_1}(\R^d)$ \quad iff \quad $p_0=p_1$, \; $q_0=q_1=2$, \; $r_0=0$, \; \text{and} \;  $s_0=s_1$.
    
    \item $H_{p_0,q_0}^{r_0,s_0}B(\R^d)=B_{p_1,q_1}^{s_1}(\R^d)$ \quad iff \quad $p_0=p_1=q_0=q_1=2$, \; $r_0=0$, \; \text{and} \; $s_0=s_1$.
    
    \item $H^{r_0,s_0}_{p_0,q_0}F(\R^d)=S^{r_1}_{p_1,q_1}F(\R^d)$ \quad iff \quad $p_0=p_1$, \; $q_0=q_1$, \; $r_0=r_1$, \; \text{and} \; $s_0=0$.
    
    \item $H^{r_0,s_0}_{p_0,q_0}B(\R^d)=S^{r_1}_{p_1,q_1}B(\R^d)$ \quad iff \quad $p_0=p_1$, \; $q_0=q_1$, \; $r_0=r_1$, \; \text{and} \; $s_0=0$.
\end{enumerate}
Especially, there holds 
$$
    H^s_p(\R^d) = H^{0,s}_{p,2}F(\R^d)
    \quad\text{and}\quad 
    S^r_p H(\R^d) = H^{r,0}_{p,2}F(\R^d), \qquad r,s\in\R,\;1<p<\infty.
$$
Moreover, all statements remain valid if $\R^d$ is replaced by some domain $\Omega$.
\end{prop}
\begin{proof}
Per definition $h_{p,q}^{r,s}x$ discretizes $H_{p,q}^{r,s}X(\R^d)$. On the other hand, according to  
\cite[Theorem~4.6 and Remark~7.2]{SchUllVed21} as well as
\cite[Theorem~2.12]{Vyb04}, they also describe the classical function spaces mentioned in \emph{(i.)}--\emph{(iv.)}\
provided that the stated conditions are fulfilled. 
Furthermore, it is well-known that these spaces are different for different parameters.
\end{proof}

For our subsequent analysis we will need sequence spaces associated to hybrid smoothness spaces $H_{p,q}^{r,s}X(\Omega)$ on domains $\Omega\subset\R^d$. 
Therefore, given $\bj\in\N_0^d$, we let
$$
    \mathfrak{D}_{\bj}:=\left\{\bk\in\Z^d\sep \supp(\psi^{\bj,\bk})\cap \Omega\neq\emptyset\right\}
    \quad\text{and}\quad 
    \nabla:=\left\{\lambda=(\bj,\bk)\in \N_0^d\times \Z^d \sep \bk\in \mathfrak{D}_{\bj}\right\}.
$$
If $\Omega$ is bounded and contains (a scaled and shifted version of) the unit cube $[0,1]^d$, we obviously have 
\begin{align*}
    \abs{\mathfrak{D}_{\bj}} \sim 2^{\abs{\bj}_1}, \qquad \bj\in\N_0^d.
\end{align*}
Based on this assumption we modify the above \autoref{defi:spaces} in the following way:
\begin{defi}\label{def:sequence_space}
    For $x\in\{b,f\}$, $0<p,q\leq \infty$ (with $p<\infty$ if $x=f$), and $r,s\in\R$, we define hybrid sequence spaces $h_{p,q}^{r,s}x(\nabla)$ as in \autoref{defi:spaces} with $\Z^d$ being replaced by $\mathfrak{D}_{\bj}$.
\end{defi}

\subsection{Quantities of interest}
In the course of this paper we shall study the asymptotic rate of convergence of the following three quantities as the number $m$ of degrees of freedom tends to infinity. 
\newpage
\begin{defi}\label{def:QoI}
	Let $A,B$ be quasi-Banach spaces, $I\in \mathcal{L}(A,B)$, 
	and $D:=\{b^\lambda\in B \sep \lambda\in \Lambda\}$ be a dictionary indexed by $\lambda\in\Lambda$. For $m\in\N_0$ we define
	\begin{enumerate}[label=(\roman*.)]
		\item the best $m$-term approximation width
		$$
    		\sigma_m(I; D)
    		:= \sigma_m(I \colon A\to B; D) 
    		:= \sup_{\norm{a\sep A}\leq 1} \inf_{\substack{\Lambda_m \subset \Lambda,\\ \abs{\Lambda_m}\leq m}} \inf_{\substack{c_\lambda\in\C,\\ \lambda\in\Lambda_m}} \bigg\| Ia- \sum_{\lambda\in\Lambda_m} c_\lambda\, b^\lambda \;\bigg|\; B \bigg\|, 
		$$
		
		\item the $m$-th Kolmogorov dictionary width
		$$
    		d_m(I;D)
    		:= d_m(I \colon A\to B;D) 
    		:= \inf_{\substack{\Lambda_m \subset \Lambda,\\ \abs{\Lambda_m}\leq m}} \sup_{\norm{a\sep A}\leq 1} \inf_{\substack{c_\lambda\in\C,\\ \lambda\in\Lambda_m}} \bigg\| Ia- \sum_{\lambda\in\Lambda_m} c_\lambda\, b^\lambda \;\bigg|\; B \bigg\|, 
		$$
		
		\item the $m$-th non-adaptive algorithm width
		$$
    		\zeta_m(I;D)
    		:= \zeta_m(I \colon A\to B;D) 
    		:= \inf_{\substack{\Lambda_m \subset \Lambda,\\ \abs{\Lambda_m}\leq m}} \inf_{\substack{c_\lambda\colon A\to\C,\\ \lambda\in\Lambda_m}} \sup_{\norm{a\sep A}\leq 1} \bigg\| Ia- \sum_{\lambda\in\Lambda_m} c_\lambda(a)\, b^\lambda \;\bigg|\; B \bigg\|
		$$
    \end{enumerate}
    w.r.t.\ the dictionary $D$.
\end{defi}
Since for best $m$-term approximation $\Lambda_m$ and $c_\lambda$ may depend on the input $a$ in an arbitrary way, $\sigma_m(I;D)$ reflects how well each individual~$Ia$ can be approximated using an optimal linear combination of at most $m$ dictionary elements. Clearly, the collection of all such approximants forms a highly non-linear manifold in the target space $B$. Hence, $\sigma_m(I;D)$ serves as a benchmark for the performance of optimal \emph{adaptive algorithms} based on $D$.
In contrast, the Kolmogorov dictionary widths $d_m(I;D)$ measure the worst case error of approximation within non-adaptively chosen optimal linear subspaces in $B$ spanned by at most $m$ dictionary elements.
Finally, $\zeta_m(I;D)$ describes the performance of optimal algorithms that are allowed to evaluate  $m$ optimal \emph{non-adaptively} chosen functionals on the input and compose these pieces of information in a linear way with a fixed collection of no more than $m$ dictionary elements to form an output. Note that these functionals neither have to be linear nor continuous.

\begin{remark}\label{rem:widths}
    Let us add some further comments which will be useful later on:
    \begin{enumerate}[label=(\roman*.)]
        \item From \autoref{def:QoI} it is obvious that all three quantities are monotonically non-increasing in $m$ and satisfy
        \begin{align}\label{eq:width_ordering}
            \sigma_m(I;D) \leq d_m(I;D) \leq \zeta_m(I;D), \qquad m\in\N_0.
        \end{align}
        
        \item All quantities defined above are based on a dictionary $D$ which has to be fixed in advance. For the application we have in mind this allows for a fair comparison of adaptive and non-adaptive \emph{wavelet} algorithms by choosing $D:=\Psi$ later on.
        
        \item Let us stress that $d_m(I;D)$ upper bounds the classical Kolmogorov $m$-widths
		$$
        	d_m(I)
        	:= d_m(I\colon A\to B)
        	:=\inf_{\substack{V\subset B\text{ linear},\\ \dim(V)\leq m}} \sup_{\norm{a\sep A}\leq 1} \inf_{v \in V} \norm{Ia - v \sep B}, \qquad m\in\N_0,
		$$
		whose convergence to zero is known to characterize the compactness of $I\in\mathcal{L}(A;B)$.
		
		\item The sequence of best $m$-term widths yields a so-called \emph{pseudo-$s$-scale} as introduced by Pietsch~\cite[Chapter~12]{Pie1980} and hence satisfies the multiplicativity assertion
		\begin{align}\label{eq:multiplicativity}
            \sigma_m\big(S\circ I \circ T; S(D)\big) \leq \norm{T \sep \mathcal{L}(Z,A)}\, \sigma_m(I;D) \, \norm{S \sep \mathcal{L}(B,C)}, \qquad m\in\N_0,
        \end{align}
        as well as the (pre)additivity
    	\begin{align}\label{eq:additivity}
            \sigma_{m_1+m_2}\big( I +J,D\big) \lesssim  \sigma_{m_1}\big( I,D\big) +  \sigma_{m_2}\big( J,D\big), \qquad m_1,m_2\in\N_0,
        \end{align}
        for all quasi-Banach spaces $A,B, C, Z$, as well as operators $I,J\in\mathcal{L}(A,B)$, $T\in\mathcal{L}(Z,A)$, and $S\in\mathcal{L}(B,C)$, respectively; see, e.g., \cite[Lemma~6.1]{Byr2018} and the references therein.
    \end{enumerate}
\end{remark}

\subsection{Reduction to sequence spaces and their embeddings}
One of the main tools in our analysis is given by the next \autoref{prop:lifting} which allows to lift results for hybrid sequence spaces (see \autoref{def:sequence_space}) to the level of function spaces introduced in \autoref{defi:spaces}. 
For the convenience of the reader, a detailed proof is given in \autoref{subsec:Proofs} below.

\begin{prop}\label{prop:lifting}
For $d\in\N$ let $X,Y\in\{B,F\}$ and $0<p_0,p_1,q_0,q_1\leq \infty$ (with $p_0<\infty$ if $X=F$ and $p_1<\infty$ if $Y=F$, respectively), as well as $r_0,r_1,s_0,s_1\in\R$.
Then the embedding
$$
    \Id\colon H^{r_0,s_0}_{p_0,q_0}X(\Omega) \to H^{r_1,s_1}_{p_1,q_1}Y(\Omega)
$$
is continuous if and only if the same holds true for $\id\colon h^{r_0,s_0}_{p_0,q_0}x(\nabla) \to h^{r_1,s_1}_{p_1,q_1}y(\nabla)$. In this case,
$$
    \sigma_m (\Id;\Psi) \sim \sigma_m (\id;E)
    \quad \text{and} \quad  
    d_m (\Id;\Psi) \sim \zeta_m (\Id;\Psi) \sim \zeta_m(\id;E) = d_m(\id;E), \qquad m\in\N_0,
$$
with dictionaries $\Psi:=\left\{\psi^\lambda\vert_\Omega \sep \lambda\in\nabla\right\}$ and $E:=\{\be^{\lambda} \sep \lambda\in\nabla\}$ 
consisting of hyperbolic wavelets and unit vectors, respectively.
\end{prop}

At the level of sequence spaces, continuous embeddings can be easily proven by standard techniques. We omit details.
\begin{lemma}[Continuous embeddings]\label{lem:embeddings}
Let $d\in\N$, as well as $x,y\in\{b,f\}$. Further let $0<p,p_0,p_1,q,q_0,q_1\leq \infty$ (with finite integrability for $f$-spaces), $r,r_0,r_1,s,s_0,s_1\in \R$ and
\begin{align}\label{eq:alphabeta}
    \alpha := r_0-r_1-\left(\frac{1}{p_0}-\frac{1}{p_1}\right)_+ \quad\textnormal{and}\quad \beta:=s_1-s_0.
\end{align}
\begin{enumerate}[label=(\roman*.)]
    \item Change of fine parameter: There holds
        $$
            h^{r,s}_{p,q_0}x(\nabla) \hookrightarrow h^{r,s}_{p,q_1}x(\nabla)
            \qquad \text{if and only if}\qquad q_0\leq q_1.
        $$

    \item Change of type: If $p<\infty$, then
        $$
            h^{r,s}_{p,\min(p,q)}b(\nabla) \hookrightarrow h^{r,s}_{p,q}f(\nabla) \hookrightarrow h^{r,s}_{p,\max(p,q)}b(\nabla).
        $$
    
    \item Change of integrability and/or smoothness I: We have 
        $$
            h^{r_0,s_0}_{p_0,q_0}x(\nabla) \hookrightarrow h^{r_1,s_1}_{p_1,q_1}y(\nabla)\qquad \text{if}\qquad  \alpha\geq 0>\beta\quad \text{or}\quad \alpha > \beta\geq 0 \quad \text{or}\quad 0>\alpha d > \beta.
        $$

    \item Change of integrability and/or smoothness II: Let $\alpha = \beta\geq 0$ or $0>\alpha d = \beta$. Then
        $$
            h^{r_0,s_0}_{p_0,q_0}b(\nabla) \hookrightarrow h^{r_1,s_1}_{p_1,q_1}b(\nabla)
            \qquad \text{if and only if}\qquad q_0\leq q_1.
        $$
\end{enumerate}
\end{lemma}

We close this section with some final remarks.
\begin{remark}\label{rem:embedding}\ 
\begin{enumerate}[label=(\roman*.)]
    \item The proof of \autoref{lem:embeddings} shows that all stated embeddings remain valid if the spaces $h^{r,s}_{p,q}x(\nabla)$ are replaced by corresponding spaces $h^{r,s}_{p,q}x$ associated to function spaces on~$\R^d$ provided that we additionally assume $p_0\leq p_1$.
    
    \item Using \autoref{lem:embeddings}(i) and (ii), simple examples show that 
    $$
        h^{r_0,s_0}_{p_0,q_0}x(\nabla) \hookrightarrow h^{r_1,s_1}_{p_1,q_1}y(\nabla)\qquad \text{only if}
        \qquad
        \begin{cases}
            \alpha \geq \beta & \text{if } \alpha > 0,\\
            \alpha d \geq \beta, & \text{else.}
        \end{cases}
    $$
    In cases of equality further restrictions on the fine parameters might come into play; see, e.g., \autoref{lem:embeddings}(iv).
    If, in addition, $p_0\leq p_1$, these cases can be viewed as generalized Sobolev embeddings, since they particularly cover the classical statements for the ranges of purely isotropic and dominating mixed smoothness spaces, respectively.
    
    \item Finally, let us add some interpretation on the quantities $\alpha$ and $\beta$ in \eqref{eq:alphabeta}. Obviously, $\beta>0$ is equivalent to a \emph{gain of isotropic smoothness}. In contrast, $\alpha>0$ refers to a \emph{loss of dominating mixed regularity}, regardless of the integrability parameters involved.
    In combination, this \emph{break-of-scale} trade-off is exactly the situation we are faced with in applications, where the PDE solutions we like to approximate are known to possess dominating mixed smoothness while errors have to be measured in isotropic energy spaces like $H^1(\Omega)$; cf.~\autoref{thm:mainsoboloev}.
\end{enumerate}
\end{remark}

\section{Approximation rates in hybrid sequence spaces}\label{sec:seqapprox}

In this section, we investigate the asymptotic decay of best $m$-term and $m$-th Kolmogorov dictionary widths, respectively, of the embedding
$$
    h^{r_0,s_0}_{p_0,q_0}x(\nabla) \hookrightarrow h^{r_1,s_1}_{p_1,q_1}y(\nabla),
    \qquad \text{where} \qquad
    r_0-r_1 - \left( \frac{1}{p_0}-\frac{1}{p_1} \right)_+ > s_1-s_0 > 0,
$$
of hybrid sequence spaces w.r.t.\ the dictionary $E:=\{\be^{\lambda} \sep \lambda\in\nabla\}$ consisting of unit vectors.
Note that according to \autoref{lem:embeddings} there are more possibilities for continuous embeddings. 
However, our methods of proof seem to be limited to this most interesting situation; see \autoref{rem:embedding}(iii). 
So, we leave the remaining cases open for further research.

\subsection{Lower bounds}
In order to derive lower bounds for our quantities of interest, we use two different arguments. 
For best $m$-term widths we employ a factorization technique that allows us to make use of results for embeddings of classical Besov sequence spaces in $d=1$ stated in \cite{DahNovSic2006}.
In contrast, for Kolmogorov dictionary widths we explicitly construct fooling sequences.

\begin{prop}[Lower bound, non-linear]\label{prop:lowerbound}
    Let $d\in\N$ and $x,y\in\{b,f\}$. Further assume $0<p_0,p_1,q_0,q_1 \leq \infty$ (with $p_0<\infty$ if $x=f$ and $p_1<\infty$ if $y=f$, respectively), and $r_0,r_1,s_0,s_1\in\R$ such that 
    \begin{align}\label{eq:cond_parameter}
        r_0-r_1 - \left( \frac{1}{p_0}-\frac{1}{p_1} \right)_+ > s_1-s_0 > 0.
    \end{align}
    Then for all $m\geq m_0$ there holds
    $$
        \sigma_m \!\left(\id \colon h^{r_0,s_0}_{p_0,q_0}x(\nabla) \to h^{r_1,s_1}_{p_1,q_1}y(\nabla); E\right)
        \gtrsim m^{-[(r_0-r_1)-(s_1-s_0)]}.
    $$
\end{prop}
\begin{proof}
    Let us consider the case $x=y=b$ first. Let $h^{r,s}_{p,q}b(\widehat{\nabla})$ denote the subspace of all $\ba \in h^{r,s}_{p,q}b(\nabla)$ such that $a_{\bj,\bk}=0$ if $j_i\neq 0$ for some $i=2,\ldots,d$. 
    Then, setting $\gamma:=r+s-1/2$, it is obvious that $h^{r,s}_{p,q}b(\widehat{\nabla})$ is isometrically isomorphic to the space $\mathbcal{b}^{\gamma}_{p,q}$ of complex sequences $\bc = (c_{\nu,k})_{\nu\in\N_0,k \in M_\nu}$ quasi-normed by 
$$
    \norm{\bc \sep \mathbcal{b}^{\gamma}_{p,q}} 
    := \left( \sum_{\nu\in\N_0} 2^{q(\gamma+1/2-1/p)\nu} \left[ \sum_{k \in M_\nu} \abs{c_{\nu,k}}^p \right]^{q/p} \right)^{1/q}
$$
with usual modifications for $\max\{p,q\}=\infty$ and $\abs{M_\nu}\sim 2^{\nu}$.
Therefore, there exist canonical universal linear restriction and extension operators,
$$
    \mathrm{re} \colon h^{r,s}_{p,q}b(\nabla) \to \mathbcal{b}^{r+s-1/2}_{p,q}
    \qquad \text{and} \qquad 
    \mathrm{ex} \colon \mathbcal{b}^{r+s-1/2}_{p,q} \to h^{r,s}_{p,q}b(\nabla),
$$
respectively, with norms bounded by one and $\mathrm{re}(E)$ being the set of unit vectors in $\mathbcal{b}^{\gamma}_{p,q}$.
The spaces $\mathbcal{b}^{\gamma}_{p,q}$ arise in the context of  (wavelet) discretizations of classical Besov spaces on bounded intervals and hence embeddings as well as corresponding approximation widths are known \cite{DahNovSic2006}. 
In particular, for $\gamma_0,\gamma_1\in\R$ and $0<p_0,p_1,q_0,q_1 \leq \infty$ we have
$$
    \mathbcal{b}^{\gamma_0}_{p_0,q_0}
    \hookrightarrow \mathbcal{b}^{\gamma_1}_{p_1,q_1}
    \qquad \text{if} \quad \gamma_0-\gamma_1 > \left( \frac{1}{p_0} - \frac{1}{p_1} \right)_+.
$$

\begin{figure}[t]
    \begin{center} 
        \begin{tikzpicture}
            \draw node at (-0.5,2) {$h^{r_0,s_0 }_{p_0,q_0}b(\widehat{\nabla}) \cong \mathbcal{b}^{\gamma_0}_{p_0,q_0}$};
            \draw node at (3,2) {$h^{r_0,s_0}_{p_0,q_0}b(\nabla)$};
            \draw node at (6,2) {$h^{r_0,s_0}_{p_0,\widetilde{q_0}}x(\nabla)$};
            
            \draw node at (-0.5,0) {$h^{r_1,s_1}_{p_1,q_1}b(\widehat{\nabla}) \cong \mathbcal{b}^{\gamma_1}_{p_1,q_1}$};
            \draw node at (3,0) {$h^{r_1,s_1}_{p_1,q_1}b(\nabla)$};
            \draw node at (6,0) {$h^{r_1,s_1}_{p_1,\widetilde{q_1}}y(\nabla)$};
            
            \draw node at (1.5,2) {$\longrightarrow$};
            \draw node at (1.5,2.4) {$\mathrm{ex}$};
            \draw node at (4.5,2) {$\hookrightarrow$};
            \draw node at (4.5,2.4) {$\mathrm{emb}_0$};
            
            \draw node at (0.5,1) {$\hookdownarrow\, \widehat{\mathrm{id}}$};
            \draw node at (3,1) {$\hookdownarrow\, \id$};
            \draw node at (6,1) {$\hookdownarrow\, \widetilde{\id}$};
            
            \draw node at (1.5,0) {$\longleftarrow$};
            \draw node at (1.5,0.4) {$\mathrm{re}$};
            \draw node at (4.5,0) {$\hookleftarrow$};
            \draw node at (4.5,0.4) {$\mathrm{emb}_1$};
        \end{tikzpicture}
    \end{center}
    \caption{Factorizations $\id=\mathrm{emb}_1\circ \widetilde{\id} \circ \mathrm{emb}_0$ and $\widehat{\mathrm{id}} = \mathrm{re}\circ \id \circ\, \mathrm{ex}$.}
    \label{fig:fact_btob}
\end{figure}

If now $\id \colon h^{r_0,s_0}_{p_0,q_0}b(\nabla) \to h^{r_1,s_1}_{p_1,q_1}b(\nabla)$ denotes our embedding of interest with parameters satisfying \eqref{eq:cond_parameter} and we let $\gamma_i:=r_i+s_i-1/2$, $i=0,1$, then we obtain the factorization
$$
    \mathrm{re} \circ \id \circ \,\mathrm{ex} 
    = \widehat{\mathrm{id}}
    \colon \mathbcal{b}^{\gamma_0}_{p_0,q_0} \to \mathbcal{b}^{\gamma_1}_{p_1,q_1},
$$
see \autoref{fig:fact_btob}.
Thus, the multiplicativity \eqref{eq:multiplicativity} yields
that lower bounds for $\sigma_m(\mathrm{id}; E)$ are obtained from those for $\sigma_m(\widehat{\mathrm{id}}; \mathrm{re}(E))$ stated in \cite[Theorem~7]{DahNovSic2006} (with $d=1$).
This proves the assertion for $x=y=b$.

Note that the result we derived so far does not depend on the fine parameters~$q_i$. Hence, for the general case of embeddings $\widetilde{\id} \colon h^{r_0,s_0}_{p_0,\widetilde{q_0}}x(\nabla) \to h^{r_1,s_1}_{p_1,\widetilde{q_1}}y(\nabla)$ we can choose suitable $q_i$, as well as embeddings $\mathrm{emb}_i$ such that 
$\mathrm{emb}_1 \circ \widetilde{\id} \circ \, \mathrm{emb}_0=\id$, see \autoref{lem:embeddings}(i),(ii), and \autoref{fig:fact_btob} again. 
So, $\sigma_m(\widetilde{\mathrm{id}}; E) \gtrsim \sigma_m(\mathrm{id}; E)$ which completes the proof in view of our previous considerations. 
\end{proof}

\begin{remark}\label{rem:lowerbound}
Some comments are in order:
\begin{enumerate}[label=(\roman*.)]
    \item If $d=1$, the presented arguments can also be used to prove matching upper bounds, while for $d>1$ this approach fails in view of the lack of surjectivity of~$\mathrm{ex}$.
    
    \item Although \autoref{prop:lowerbound} holds true for $s_1=s_0$ as well, there are good reasons to assume that sharp lower bounds should contain additional log terms if $d>1$.
    
    \item In arbitrary dimensions our proof technique allows to derive lower bounds for all other pseudo-$s$-numbers of $\id$ (such as, e.g., entropy numbers or Gelfand widths) from known results for $\widehat{\mathrm{id}}$, too. In particular, for classical Kolmogorov $m$-widths as discussed in \autoref{rem:widths}, \cite[Theorem~4.6]{Vyb2008} implies
    $$
        d_m(\id) \gtrsim  m^{-[(r_0-r_1)-(s_1-s_0)-(1/p_0-1/p_1)_+]}, \qquad m\geq m_0, 
    $$
    if, in addition to the assumptions of \autoref{prop:lowerbound}, there holds  $0<p_0\leq p_1 \leq 2$ or $p_1\leq p_0$.
\end{enumerate}
\end{remark}

If we restrict ourselves to the larger Kolmogorov \emph{dictionary} widths $d_m(\id;E)$, we can get rid of additional assumptions on the relation of the integrability parameters.
\begin{prop}[Lower bound, linear]\label{prop:lowerbound_linear}
    Let $d\in\N$ and $x,y\in\{b,f\}$. Further assume $0<p_0,p_1,q_0,q_1 \leq \infty$ (with $p_0<\infty$ if $x=f$ and $p_1<\infty$ if $y=f$, respectively), and $r_0,r_1,s_0,s_1\in\R$ such that 
    \begin{align*}
        r_0-r_1 - \left( \frac{1}{p_0}-\frac{1}{p_1} \right)_+ > s_1-s_0 > 0.
    \end{align*}
	Then for all $m\geq m_0$ there holds
	$$
		d_m\!\left(\id \colon h^{r_0,s_0}_{p_0,q_0}x(\nabla) \to h^{r_1,s_1}_{p_1,q_1}y(\nabla); E\right)
		\gtrsim m^{-[(r_0-r_1)-(s_1-s_0)-(1/p_0-1/p_1)_+]}. 
	$$
\end{prop}
\begin{proof}
	If $p_1\leq p_0$ and hence $(1/p_0 - 1/p_1)_+=0$, the assertion follows from \autoref{prop:lowerbound} and \eqref{eq:width_ordering}.
	For the remaining case $p_0<p_1\leq\infty$ note that, as $d_m(\id;E)$ is monotone in $m$,
	it suffices 
	to prove the claim for all	$m:=m(M):=\ceil{c\, 2^M}$ with $M\in\N$ being large, where $c>0$ is arbitrarily fixed. 
	For each $M\in\N$ let $\bj_M^*:=(M,0,\ldots,0)$. Since by assumption there holds $\abs{\mathfrak{D}_{\bj}}\sim 2^{\abs{\bj}_1}$ for every $\bj\in\N_0^d$, we can fix $c>0$ such that 
	$$
		\abs{\mathfrak{D}_{\bj_M^*}} > m(M), \qquad M\geq M_0.
	$$
	Hence, for every given index collection $\Lambda_m \subset\nabla$ with $\abs{\Lambda_m}\leq m$ we find $\bk_M^*\in \mathfrak{D}_{\bj_M^*}$ such that $\lambda_M^*:=(\bj_M^*,\bk_M^*)\in \nabla\setminus\Lambda_m$.
	Then the fooling sequence 
	$$
		\ba^M:= C \be^{\lambda_M^*} \qquad \text{with} \qquad C := 2^{-(r_0-1/p_0+s_0)M}
	$$
	satisfies
	\begin{align*}
		\norm{\ba^M \sep h^{r_0,s_0}_{p_0,q_0}x(\nabla)}
		&\lesssim \norm{\ba^M \sep h^{r_0,s_0}_{p_0,q}b(\nabla)} 
		= 2^{(r_0-1/p_0)\abs{\bj_M^*}_1+s_0\abs{\bj_M^*}_\infty} \, C 
		= 1,
	\end{align*}
    where we used $q:=\min\{p_0,q_0\}<\infty$ and \autoref{lem:embeddings}.  
	On the other hand, the so-called lattice property (see \autoref{lem:lattice_prop} below) together with $\supp(\ba^M)\cap \Lambda_m=\emptyset$ yields
	\begin{align*}
		\inf_{\substack{c_\lambda\in\C,\\ \lambda\in\Lambda_m}} \bigg\| \ba^M- \sum_{\lambda\in\Lambda_m} c_\lambda\, \be^\lambda \,\bigg|\, h^{r_1,s_1}_{p_1,q_1}y(\nabla) \bigg\| 
		&\gtrsim \inf_{\substack{c_\lambda\in\C,\\ \lambda\in\Lambda_m}} \bigg\| \ba^M- \sum_{\lambda\in\Lambda_m} c_\lambda\, \be^\lambda \,\bigg|\, h^{r_1,s_1}_{p_1,\infty}b(\nabla) \bigg\| \\ 
		&= \bigg\| \ba^M- \sum_{\lambda\in\Lambda_m} a^M_\lambda\, \be^\lambda \,\bigg|\, h^{r_1,s_1}_{p_1,\infty}b(\nabla) \bigg\| \\
		&= \norm{\ba^M \sep h^{r_1,s_1}_{p_1,\infty}b(\nabla)} \\
		&= C\, 2^{(r_1-1/p_1+s_1)M} \\
		&\sim m^{-\big[(r_0-r_1)-(s_1-s_0)-(1/p_0 -1/p_1) \big]}. \qedhere
	\end{align*}
\end{proof}
Note that for $p_1>p_0$ the proof does not require any restriction on the parameters. 
Further, we like to mention that more advanced fooling sequences can be used to give an alternative proof of \autoref{prop:lowerbound} for certain parameter constellations.

\subsection{Upper bounds}\label{sect:upper}
We complement the lower bounds obtained by abstract arguments in the previous subsection by the analysis of some explicitly constructed optimal approximation algorithms.
In view of \autoref{prop:lowerbound} and \ref{prop:lowerbound_linear}, there is some hope that optimal linear methods already show the maximal rate of convergence $(r_0-r_1)-(s_1-s_0)$. Therefore, we treat such algorithms first, but  introduce some technicalities beforehand.

\begin{defi}\label{def:delta}
    Let $d\in\N$ and $\mathfrak{D}_{\bj}\subset\Z^d$ be as above, as well as $\alpha,\beta\in\R$. We let
    \begin{align*}
        \Delta_\mu &:= \left\{ \bj \in \N_0^d \sep \alpha \abs{\bj}_1 - \beta \abs{\bj}_\infty \leq \mu \right\}, \qquad \mu\in\N_0.
    \end{align*}
    Further we set $\mathfrak{L}_0:=\Delta_0$, as well as
    $$
        \mathfrak{L}_\mu := \Delta_\mu \setminus \Delta_{\mu-1} 
        =\left\{ \bj \in \N_0^d \sep \mu-1 < \alpha \abs{\bj}_1 - \beta \abs{\bj}_\infty \leq \mu \right\}, \qquad \mu\in\N.
    $$
    Finally, let 
    $$
        \nabla_\mu := \left\{ (\bj, \bk) \in\N_0^d\times \Z^d \sep \bj \in \mathfrak{L}_\mu,\; \bk \in \mathfrak{D}_{\bj} \right\}.
    $$
\end{defi}
We note in passing that $\N_0^d=\bigcup_{\mu=0}^\infty \mathfrak{L}_\mu$ and hence $\nabla=\bigcup_{\mu=0}^\infty \nabla_\mu$ (disjoint unions).
If we like to stress the dependence on $\alpha$ and $\beta$ in the notation, we write $\Delta_\mu(\alpha,\beta)$, $\mathfrak{L}_\mu(\alpha,\beta)$, and $\nabla_\mu(\alpha,\beta)$, respectively. 
Later on, we shall choose these parameters depending on the source and target spaces of the embedding under consideration.

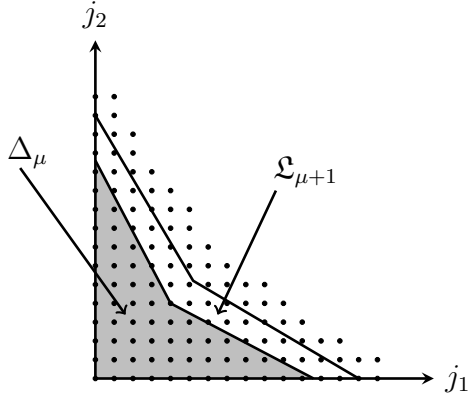
\begin{figure}[ht]
    \centering
        \begin{tikzpicture}
        \draw[fill, lightgray] rectangle (0,0) -- (0,2.9) -- (1,1) -- (2.9,0) -- (0,0);
        \draw[->, >= stealth, line width=1pt] (0,0) -- (4.5,0) node[right]{$j_1$};
        \draw[->, >= stealth, line width=1pt] (0,0) -- (0,4.5) node[above]{$j_2$};
        \draw[-, line width=1pt] (0,3.5) -- (1.3,1.3) -- (3.5,0);
        \draw[-, thick, line width=1pt] (0,2.9) -- (1,1) -- (2.9,0);
        \draw[<-, thick, line width=1pt] (0.4,0.85) -- (-1,2.8);
        \draw node at (-0.9,3) {$\Delta_{\mu}$};
        \draw[<-, thick, line width=1pt] (1.625,0.85) -- (2.4,2.5);
        \draw node at (2.8,2.75) {$\mathfrak{L}_{\mu+1}$};
        \draw[fill] (0,0) circle[radius=0.03];
        \draw[fill] (0,0.25) circle[radius=0.03];
        \draw[fill] (0,0.5) circle[radius=0.03];
        \draw[fill] (0,0.75) circle[radius=0.03];
        \draw[fill] (0,1) circle[radius=0.03];
        \draw[fill] (0,1.25) circle[radius=0.03];
        \draw[fill] (0,1.5) circle[radius=0.03];
        \draw[fill] (0,1.75) circle[radius=0.03];
        \draw[fill] (0,2) circle[radius=0.03];
        \draw[fill] (0,2.25) circle[radius=0.03];
        \draw[fill] (0,2.5) circle[radius=0.03];
        \draw[fill] (0,2.75) circle[radius=0.03];
        \draw[fill] (0,3) circle[radius=0.03];
        \draw[fill] (0,3.25) circle[radius=0.03];
        \draw[fill] (0,3.5) circle[radius=0.03];
        \draw[fill] (0,3.75) circle[radius=0.03];
        \draw[fill] (0.25,0.25) circle[radius=0.03];
        \draw[fill] (0.25,0.5) circle[radius=0.03];
        \draw[fill] (0.25,0.75) circle[radius=0.03];
        \draw[fill] (0.25,1) circle[radius=0.03];
        \draw[fill] (0.25,1.25) circle[radius=0.03];
        \draw[fill] (0.25,1.5) circle[radius=0.03];
        \draw[fill] (0.25,1.75) circle[radius=0.03];
        \draw[fill] (0.25,2) circle[radius=0.03];
        \draw[fill] (0.25,2.25) circle[radius=0.03];
        \draw[fill] (0.25,2.5) circle[radius=0.03];
        \draw[fill] (0.25,2.75) circle[radius=0.03];
        \draw[fill] (0.25,3) circle[radius=0.03];
        \draw[fill] (0.25,3.25) circle[radius=0.03];
        \draw[fill] (0.25,3.5) circle[radius=0.03];
        \draw[fill] (0.25,3.75) circle[radius=0.03];
        \draw[fill] (0.5,0.25) circle[radius=0.03];
        \draw[fill] (0.5,0.5) circle[radius=0.03];
        \draw[fill] (0.5,0.75) circle[radius=0.03];
        \draw[fill] (0.5,1) circle[radius=0.03];
        \draw[fill] (0.5,1.25) circle[radius=0.03];
        \draw[fill] (0.5,1.5) circle[radius=0.03];
        \draw[fill] (0.5,1.75) circle[radius=0.03];
        \draw[fill] (0.5,2) circle[radius=0.03];
        \draw[fill] (0.5,2.25) circle[radius=0.03];
        \draw[fill] (0.5,2.5) circle[radius=0.03];
        \draw[fill] (0.5,2.75) circle[radius=0.03];
        \draw[fill] (0.5,3) circle[radius=0.03];
        \draw[fill] (0.5,3.25) circle[radius=0.03];
        \draw[fill] (0.75,0.25) circle[radius=0.03];
        \draw[fill] (0.75,0.5) circle[radius=0.03];
        \draw[fill] (0.75,0.75) circle[radius=0.03];
        \draw[fill] (0.75,1) circle[radius=0.03];
        \draw[fill] (0.75,1.25) circle[radius=0.03];
        \draw[fill] (0.75,1.5) circle[radius=0.03];
        \draw[fill] (0.75,1.75) circle[radius=0.03];
        \draw[fill] (0.75,2) circle[radius=0.03];
        \draw[fill] (0.75,2.25) circle[radius=0.03];
        \draw[fill] (0.75,2.5) circle[radius=0.03];
        \draw[fill] (0.75,2.75) circle[radius=0.03];
        \draw[fill] (1,0.25) circle[radius=0.03];
        \draw[fill] (1,0.5) circle[radius=0.03];
        \draw[fill] (1,0.75) circle[radius=0.03];
        \draw[fill] (1,1) circle[radius=0.03];
        \draw[fill] (1,1.25) circle[radius=0.03];
        \draw[fill] (1,1.5) circle[radius=0.03];
        \draw[fill] (1,1.75) circle[radius=0.03];
        \draw[fill] (1,2) circle[radius=0.03];
        \draw[fill] (1,2.25) circle[radius=0.03];
        \draw[fill] (1,2.5) circle[radius=0.03];
        \draw[fill] (1.25,0.25) circle[radius=0.03];
        \draw[fill] (1.25,0.5) circle[radius=0.03];
        \draw[fill] (1.25,0.75) circle[radius=0.03];
        \draw[fill] (1.25,1) circle[radius=0.03];
        \draw[fill] (1.25,1.25) circle[radius=0.03];
        \draw[fill] (1.25,1.5) circle[radius=0.03];
        \draw[fill] (1.25,1.75) circle[radius=0.03];
        \draw[fill] (1.25,2) circle[radius=0.03];
        \draw[fill] (1.25,2.25) circle[radius=0.03];
        \draw[fill] (1.5,0.25) circle[radius=0.03];
        \draw[fill] (1.5,0.5) circle[radius=0.03];
        \draw[fill] (1.5,0.75) circle[radius=0.03];
        \draw[fill] (1.5,1) circle[radius=0.03];
        \draw[fill] (1.5,1.25) circle[radius=0.03];
        \draw[fill] (1.5,1.5) circle[radius=0.03];
        \draw[fill] (1.5,1.75) circle[radius=0.03];
        \draw[fill] (1.75,0.25) circle[radius=0.03];
        \draw[fill] (1.75,0.5) circle[radius=0.03];
        \draw[fill] (1.75,0.75) circle[radius=0.03];
        \draw[fill] (1.75,1) circle[radius=0.03];
        \draw[fill] (1.75,1.25) circle[radius=0.03];
        \draw[fill] (1.75,1.5) circle[radius=0.03];
        \draw[fill] (0.25,0) circle[radius=0.03];
        \draw[fill] (0.5,0) circle[radius=0.03];
        \draw[fill] (0.75,0) circle[radius=0.03];
        \draw[fill] (1,0) circle[radius=0.03];
        \draw[fill] (1.25,0) circle[radius=0.03];
        \draw[fill] (1.5,0) circle[radius=0.03];
        \draw[fill] (1.75,0) circle[radius=0.03];
        \draw[fill] (2,0) circle[radius=0.03];
        \draw[fill] (2.25,0) circle[radius=0.03];
        \draw[fill] (2.5,0) circle[radius=0.03];
        \draw[fill] (2.75,0) circle[radius=0.03];
        \draw[fill] (3,0) circle[radius=0.03];
        \draw[fill] (3.25,0) circle[radius=0.03];
        \draw[fill] (3.5,0) circle[radius=0.03];
        \draw[fill] (3.75,0) circle[radius=0.03]; 
        \draw[fill] (2,0.25) circle[radius=0.03];
        \draw[fill] (2.25,0.25) circle[radius=0.03];
        \draw[fill] (2.5,0.25) circle[radius=0.03];
        \draw[fill] (2.75,0.25) circle[radius=0.03];
        \draw[fill] (3,0.25) circle[radius=0.03];
        \draw[fill] (3.25,0.25) circle[radius=0.03];
        \draw[fill] (3.5,0.25) circle[radius=0.03];
        \draw[fill] (3.75,0.25) circle[radius=0.03];
        \draw[fill] (2,0.5) circle[radius=0.03];
        \draw[fill] (2.25,0.5) circle[radius=0.03];
        \draw[fill] (2.5,0.5) circle[radius=0.03];
        \draw[fill] (2.75,0.5) circle[radius=0.03];
        \draw[fill] (3,0.5) circle[radius=0.03];
        \draw[fill] (3.25,0.5) circle[radius=0.03];
        \draw[fill] (2,0.75) circle[radius=0.03];
        \draw[fill] (2.25,0.75) circle[radius=0.03];
        \draw[fill] (2.5,0.75) circle[radius=0.03];
        \draw[fill] (2.75,0.75) circle[radius=0.03];
        \draw[fill] (2,1) circle[radius=0.03];
        \draw[fill] (2.25,1) circle[radius=0.03];
        \draw[fill] (2.5,1) circle[radius=0.03];
        \draw[fill] (2,1.25) circle[radius=0.03];
        \draw[fill] (2.25,1.25) circle[radius=0.03];
\end{tikzpicture}
    \caption{Illustration of $\Delta_{\mu}$ and $\mathfrak{L}_{\mu+1}$ for some $\alpha>\beta>0$ in $d=2$.}
    \label{fig:layer}
\end{figure}

A visualization of the just defined quantities can be found in \autoref{fig:layer}. 
Therein, the ``kink'' at the boundary of the shaded area $\Delta_{\mu}$ is caused by the interplay of the different vector norms involved and the fact that $\beta>0$ which distinguishes our setting of interest from the purely dominating mixed situation. 

Some combinatorics related to the quantities from \autoref{def:delta} may be found in \autoref{sect:combinatorics} below. 
At this point, we shall restrict ourselves to the following observation.
\begin{remark}\label{rem:counting}
If $\alpha > \beta >0$, then $\abs{\mathfrak{D}_{\bj}}\sim 2^{\abs{\bj}_1}$ and \autoref{lem:counting} applied for $\delta:=1$ yield
\begin{align*}
    \abs{\nabla_\mu} 
    \sim \sum_{\bj \in \Delta_\mu} 2^{\abs{\bj}_1} - \sum_{\bj \in \Delta_{\mu-1}} 2^{\abs{\bj}_1}
    &\sim 2^{\mu /(\alpha-\beta)} - 2^{(\mu-1) /(\alpha-\beta)} 
    \sim 2^{\mu /(\alpha-\beta)}
\end{align*}
for all $\mu\geq \alpha-\beta$ $(>0)$ with constants independent of $\mu$.
\end{remark}

\subsubsection{Linear approximation}
Let us consider linear algorithms $\A_M$, $M\in\N_0$, of the form
\begin{align}\label{eq:AL}
    \A_M\ba := \sum_{\bj\in\Delta_M} \sum_{\bk\in\mathfrak{D}_{\bj}} a_{\bj,\bk}\,\be^{\bj,\bk},
    \qquad
    \ba = \sum_{\bj\in\N_0^d}\sum_{\bk\in\mathfrak{D}_{\bj}} a_{\bj,\bk}\,\be^{\bj,\bk},
\end{align}
where $\Delta_M=\Delta_M(\alpha,\beta)$ is defined as above and $\be^{\bj,\bk}$ denote the respective unit vectors.
In this way, $\A_M$ takes into account the complete first $M+1$ layers $\mathfrak{L}_\mu$ of resolution vectors $\bj\in\N_0^d$ which contribute most to the sequence space (quasi-)norm of the input $\ba$; see \autoref{fig:layer}.

\begin{prop}[Upper bound, linear]\label{prop:linear_upperbound} 
    For $d\in\N$ let $0< p_0, p_1, q_1 \leq \infty$, as well as $r_0,r_1,s_0,s_1 \in\R$ such that 
    \begin{align*}
        r_0-r_1 - \left( \frac{1}{p_0} - \frac{1}{p_1} \right)_+ > s_1-s_0 > 0.
    \end{align*}
    Then there exist constants $M_0, c_1, c_2>0$ and a sequence $(\A_M)_{M\in\N}$ of linear algorithms such that for all $M\geq M_0$ and $m:=m(M):=\ceil{c_1 \, 2^{M/[(r_0-r_1)-(s_1-s_0)-(1/p_0-1/p_1)_+]}}$ there holds that
    $$
        \norm{\ba-\A_M\ba \sep h_{p_1,q_1}^{r_1,s_1}b(\nabla)} 
        \leq c_2 \, m^{-[(r_0-r_1)-(s_1-s_0)-(1/p_0-1/p_1)_+]}\! \norm{\ba \sep h_{p_0,\infty}^{r_0,s_0}b(\nabla)}\!, \quad\ba \in h_{p_0,\infty}^{r_0,s_0}b(\nabla),
    $$
    and $\A_M$ uses at most $m$ degrees of freedom.
\end{prop}
\begin{proof}
It suffices to consider the case $p_1 \geq p_0$ since otherwise \autoref{lem:embeddings}(iv) allows to reduce to this situation as follows:
$$
    \norm{\ba-\A_M\ba \sep h_{p_1,q_1}^{r_1,s_1}b(\nabla)} 
    \lesssim \norm{\ba-\A_M\ba \sep h_{p_0,q_1}^{r_1,s_1}b(\nabla)} 
    \lesssim m^{-[(r_0-r_1)-(s_1-s_0)]} \,\norm{\ba \sep h_{p_0,\infty}^{r_0,s_0}b(\nabla)} \quad 
$$
for all $\ba \in h_{p_0,\infty}^{r_0,s_0}b(\nabla)$.

So let $p_1 \geq p_0$. In addition, we may assume that $p_1,q_1<\infty$ as the remaining cases can be obtained by obvious modifications. 
Choose $\epsilon \in (0,s_1-s_0)$ and define
$$
    \alpha := r_0-r_1 - \left( \frac{1}{p_0} - \frac{1}{p_1} \right) - \epsilon 
    \qquad\text{as well as}\qquad 
    \beta:=s_1-s_0-\epsilon
$$
such that $\alpha > \beta > 0$. 
Based on this, we let $\Delta_M:=\Delta_M(\alpha,\beta)$ and consider the linear algorithms~$\A_M$, $M\in\N_0$, as defined in \eqref{eq:AL}. 
If $M\geq M_0:=\ceil{\alpha-\beta}$, then $\A_M$ uses
$$
    \sum_{\bj\in\Delta_M} \abs{\mathfrak{D}_{\bj}} 
    \sim \sum_{\bj\in\Delta_M} 2^{\abs{\bj}_1}
    \sim 2^{M/[\alpha-\beta]} 
    = 2^{M/[(r_0-r_1)-(s_1-s_0)-(1/p_0-1/p_1)_+]}
$$
degrees of freedom, where the implied constants are independent of $M$; see \autoref{lem:counting}.
Given $\ba \in h_{p_0,\infty}^{r_0,s_0}b(\nabla)$, for its error $E_1(\ba):=\ba-\A_M\ba$ there holds
\begin{align*}
    &\norm{E_1(\ba)\sep h_{p_1,q_1}^{r_1,s_1}b(\nabla)} \\
    &\qquad= \left[ \sum_{\bj \in (\Delta_M)^C}  2^{q_1 \!\big[(r_1-1/p_1)\abs{\bj}_{1}+s_1\abs{\bj}_{\infty}\big]} \left( \sum_{\bk \in\mathfrak{D}_{\bj}} \abs{a_{\bj,\bk}}^{p_1} \right)^{q_1/p_1} \right]^{1/q_1} \\
    &\qquad = \left[\rule{0pt}{28pt}  \sum_{\mu= M+1}^\infty \sum_{\bj \in \mathfrak{L}_\mu} 2^{-q_1 \!\left[\big((r_0-r_1)-(1/p_0-1/p_1)-\epsilon\big) \abs{\bj}_1 - (s_1-s_0-\epsilon) \abs{\bj}_{\infty}\right]} \, \right.\\
        &\qquad\qquad\qquad\qquad\qquad \cdot \left. 2^{-q_1\epsilon(\abs{\bj}_1-\abs{\bj}_\infty)} \, 2^{q_1\!\big[ (r_0-1/p_0)\abs{\bj}_1 + s_0\abs{\bj}_\infty \big]} \left( \sum_{\bk\in\mathfrak{D}_{\bj}} \abs{a_{\bj,\bk}}^{p_1} \right)^{q_1/p_1} \right]^{1/q_1} \\
    &\qquad\sim \left[ \sum_{\mu= M+1}^\infty 2^{-q_1\mu} \sum_{\bj\in \mathfrak{L}_{\mu}} 2^{-q_1\epsilon(\abs{\bj}_1-\abs{\bj}_\infty)} \, 2^{q_1\!\big[ (r_0-1/p_0)\abs{\bj}_1 + s_0\abs{\bj}_\infty \big]} \left( \sum_{\bk\in\mathfrak{D}_{\bj}} \abs{a_{\bj,\bk}}^{p_1} \right)^{q_1/p_1} \right]^{1/q_1}.
\end{align*}
Then $p_0\leq p_1$ shows that
\begin{align*}
    &\norm{E_1(\ba)\sep h_{p_1,q_1}^{r_1,s_1}b(\nabla)} \\
    &\qquad \lesssim \left[ \sum_{\mu= M+1}^\infty 2^{-q_1\mu} \sum_{\bj\in \mathfrak{L}_{\mu}} 2^{-q_1\epsilon(\abs{\bj}_1-\abs{\bj}_\infty)} \left( 2^{ (r_0-1/p_0)\abs{\bj}_1 + s_0\abs{\bj}_\infty} \left[ \sum_{\bk\in\mathfrak{D}_{\bj}} \abs{a_{\bj,\bk}}^{p_0} \right]^{1/p_0} \right)^{q_1} \right]^{1/q_1} \\
    &\qquad \leq \left[ \sum_{\mu= M+1}^\infty 2^{-q_1\mu} \sum_{\bj\in \mathfrak{L}_{\mu}} 2^{-q_1\epsilon(\abs{\bj}_1-\abs{\bj}_\infty)} \right]^{1/q_1} \, \norm{\ba \sep h_{p_0,\infty}^{r_0,s_0}b(\nabla)}.
\end{align*}
Moreover, \autoref{lem:log-killer} yields
$$
    \sum_{\mu= M+1}^\infty 2^{-q_1\mu}
    \sum_{\bj\in \mathfrak{L}_{\mu}} 2^{-q_1\epsilon(\abs{\bj}_1-\abs{\bj}_\infty)}
    \lesssim \sum_{\mu= M+1}^\infty 2^{-q_1\mu} 
    \sim 2^{-q_1 M}
$$
such that finally
\begin{align}
    \norm{E_1(\ba)\sep h_{p_1,q_1}^{r_1,s_1}b(\nabla)}
    &\lesssim 2^{-M} \norm{\ba \sep h_{p_0,\infty}^{r_0,s_0}b(\nabla)} \label{eq:linear_error} \\
    &\sim m^{-[(r_0-r_1)-(s_1-s_0)-(1/p_0-1/p_1)_+]} \,\norm{\ba \sep h_{p_0,\infty}^{r_0,s_0}b(\nabla)}. \nonumber\qedhere
\end{align}
\end{proof}

Using the embeddings from \autoref{lem:embeddings}, we can replace the source space $h_{p_0,\infty}^{r_0,s_0}b(\nabla)$ by an arbitrary $b$- or $f$-space of same smoothness and integrability. The same applies for the target space $h_{p_1,q_1}^{r_1,s_1}b(\nabla)$. Together with the monotonicity of the Kolmogorov dictionary widths this proves
\begin{corollary}\label{cor:linear_approx}
    For $d\in\N$ let $x,y\in\{b,f\}$,  $0< p_0, p_1, q_0, q_1 \leq \infty$, and $r_0,r_1,s_0,s_1 \in\R$ s.t.
    $$
        r_0-r_1 - \left( \frac{1}{p_0} - \frac{1}{p_1} \right)_+ > s_1-s_0 > 0
    $$
    (with $p_0<\infty$ if $x=f$ and $p_1<\infty$ if $y=f$, respectively). 
    Then 
    $$
        d_m \!\left( \id \colon h^{r_0,s_0}_{p_0,q_0}x(\nabla) \to h^{r_1,s_1}_{p_1,q_1}y(\nabla); E \right) 
        \lesssim m^{-[(r_0-r_1)-(s_1-s_0)-(1/p_0-1/p_1)_+]}, \qquad m\geq m_0.
    $$
    In particular, the embedding $h^{r_0,s_0}_{p_0,q_0}x(\nabla) \hookrightarrow h^{r_1,s_1}_{p_1,q_1}y(\nabla)$ is compact.
\end{corollary}

\begin{remark}\label{rem:linear_optimal}
    \autoref{cor:linear_approx} implies several optimality statements.
    \begin{enumerate}[label=(\roman*.)]
        \item From the lattice property (\autoref{lem:lattice_prop}) it follows that on the level of sequence spaces the Kolmogorov dictionary widths $d_m(\id;E)$ dominate classical approximation numbers
        \begin{align*}
          	a_m(\id)
    		&:= a_m\!\left( \id \colon h^{r_0,s_0}_{p_0,q_0}x(\nabla) \to h^{r_1,s_1}_{p_1,q_1}y(\nabla) \right) \\
    		&:= \inf_{\substack{\A_m\in\mathcal{L}\left(h^{r_0,s_0}_{p_0,q_0}x(\nabla), h^{r_1,s_1}_{p_1,q_1}y(\nabla)\right),\\ \mathrm{rank}(\A_m) \leq m}} \sup_{\norm{\ba \sep h^{r_0,s_0}_{p_0,q_0}x(\nabla)}\leq 1} \norm{\ba -\A_m\ba \sep h^{r_1,s_1}_{p_1,q_1}y(\nabla)}
        \end{align*}
        which in turn are general upper bounds for usual Kolmogorov $m$-widths $d_m(\id)$; see \cite[Chapter~11]{Pie1980}. 
        Combined with the lower bounds mentioned in \autoref{rem:lowerbound} we thus have
        $$
            a_m(\id) \sim d_m(\id) \sim d_m(\id;E) \sim m^{-[(r_0-r_1)-(s_1-s_0)-(1/p_0-1/p_1)_+]}, \qquad m\geq m_0,
        $$
        if either $0<p_0 \leq p_1 \leq 2$ or $p_1\leq p_0$.
        
        \item Note that in view of \autoref{prop:lowerbound_linear} the rate found in \autoref{cor:linear_approx} is sharp. 
        
        \item Due to \eqref{eq:width_ordering} we can conclude from \autoref{cor:linear_approx} that
        $$
            \sigma_m \!\left( \id \colon h^{r_0,s_0}_{p_0,q_0}x(\nabla) \to h^{r_1,s_1}_{p_1,q_1}y(\nabla); E \right) 
            \lesssim m^{-[(r_0-r_1)-(s_1-s_0)-(1/p_0-1/p_1)_+]}, \qquad m\geq m_0,
        $$
        which according to \autoref{prop:lowerbound} is optimal if $p_1\leq p_0$. That is, in this regime there is no need for non-linear algorithms, since the linear approximation by $\A_L$ (as constructed in the proof of \autoref{prop:linear_upperbound}) is already best possible.
        However, this becomes false for $p_0<p_1$ as we shall see below.
    \end{enumerate}
\end{remark}

\subsubsection{Non-linear approximation}
In order to improve the speed of convergence if $p_0<p_1$, we will approximate given sequences
$$
    \ba
    =\sum_{\bj\in\N_0^d}\sum_{\bk\in\mathfrak{D}_{\bj}}a_{\bj,\bk}\,\be^{\bj,\bk}
    =\sum_{\mu=0}^{\infty} \sum_{\lambda\in\nabla_{\mu}} a_{\lambda}\,\be^{\lambda}
$$
through \emph{non-linear} algorithms $\B_M$, $M\in\N_0$, of the form
\begin{align}\label{eq:BM}
    \mathcal{B}_M(\ba)
    := \mathcal{A}_M\ba
    + \sum_{\mu=M+1}^{N_M} \sum_{\lambda\in\Lambda_{M,\mu}} a_{\lambda}\,\be^{\lambda}
\end{align}
with some $N_M>M$ and subsets $\Lambda_{M,\mu}\subseteq \nabla_{\mu}= \left\{ \lambda=(\bj, \bk) \in\N_0^d\times \Z^d \sep \bj \in \mathfrak{L}_\mu,\, \bk \in \mathfrak{D}_{\bj} \right\}$ indicating the most important coefficients of $\ba$ at layer $\mathfrak{L}_\mu$, $\mu=M+1,\ldots,N_M$.
Therein, $\Lambda_{M,\mu}:=\{\varphi_\mu(n) \sep n=1,\ldots,m_{M,\mu}\}$ with some $m_{M,\mu}\in\N$ (to be specified later) and  a bijection $\varphi_\mu \colon \big\{1,2,\ldots,\abs{\nabla_\mu} \big\} \to \nabla_\mu$ which yields a non-increasing rearrangement of the weighted coefficients portion
$$
    \big( 2^{-(\abs{\bj}_1-\abs{\bj}_\infty)\epsilon/2} \, 2^{(r_0-1/p_0) \abs{\bj}_1+ s_0\abs{\bj}_\infty} \abs{a_{\bj,\bk}} \big)_{(\bj,\bk)\in \nabla_{\mu}}
$$
with some $\epsilon >0$. 
That is,
\begin{align*}
    & 2^{-\left(\abs{\bj_{\varphi_\mu(n)}}_1-\abs{\bj_{\varphi_\mu(n)}}_\infty\right)\epsilon/2} \, 2^{(r_0-1/p_0) \abs{\bj_{\varphi_\mu(n)}}_1 + s_0\abs{\bj_{\varphi_\mu(n)}}_\infty} \abs{a_{\varphi_\mu(n)}} \\
    &\qquad \geq 2^{-\left(\abs{\bj_{\varphi_\mu(n+1)}}_1-\abs{\bj_{\varphi_\mu(n+1)}}_\infty\right)\epsilon/2} \, 2^{(r_0-1/p_0) \abs{\bj_{\varphi_\mu(n+1)}}_1 + s_0\abs{\bj_{\varphi_\mu(n+1)}}_\infty} \abs{a_{\varphi_\mu(n+1)}},
\end{align*}
where $\bj_{\varphi_\mu(n)}\in \mathfrak{L}_\mu$ denotes the projection of $\varphi_\mu(n)=(\bj_{\varphi_\mu(n)},\bk_{\varphi_\mu(n)})\in \nabla_\mu$ to its first component.
Hence, at first $\mathcal{B}_M$ takes into account the first full $M+1$ layers of resolution (linear approximation) followed by a sparse non-linear correction based on information from some subsequent layers $\mathfrak{L}_\mu$; see \autoref{fig:layer} again.

Since for $p_0\geq p_1$ the linear algorithm is already optimal, it suffices to consider $p_0<p_1$.

\begin{prop}[Upper bound, non-linear]\label{prop:nonlin_upperbound} 
    For $d\in\N$ let $0< p_0, p_1, q_1 \leq \infty$ with $p_0<p_1$, as well as $r_0,r_1,s_0,s_1 \in\R$ such that 
    \begin{align*}
        r_0-r_1 - \left( \frac{1}{p_0} - \frac{1}{p_1} \right) > s_1-s_0 > 0.
    \end{align*}
    Then there exist constants $M_0',c_1',c_2'>0$ and a sequence $(\B_M)_{M\in\N}$ of non-linear algorithms such that for all $M\geq M_0'$ and $m:=m(M):=\ceil{c_1'\,2^{M/[(r_0-r_1)-(s_1-s_0)-(1/p_0-1/p_1)]}}$ there holds
    \begin{equation*}
        \norm{\ba-\mathcal{B}_M(\ba) \sep h_{p_1,q_1}^{r_1,s_1}b(\nabla)} 
        \leq c_2' \, m^{-[(r_0-r_1)-(s_1-s_0)]} \norm{\ba \sep h_{p_0,\infty}^{r_0,s_0}b(\nabla)}, \qquad \ba \in h_{p_0,\infty}^{r_0,s_0}b(\nabla),
    \end{equation*}
    and $\mathcal{B}_M$ uses at most $m$ degrees of freedom.
\end{prop}

\begin{proof}
Once more, choose $\epsilon \in (0,s_1-s_0)$ and fix 
$$
    \alpha := r_0-r_1 - \left(\frac{1}{p_0} - \frac{1}{p_1}\right) - \epsilon
    \qquad \text{as well as}\qquad \beta := s_1-s_0 - \epsilon
$$
such that $\alpha>\beta>0$. 
Then, as $p_0<p_1$ and $M$ is supposed to be large, we have
$$
    N_M:= \floor{\frac{(r_0-r_1)-(s_1-s_0)}{\alpha-\beta}\, M} > M
$$
and we can choose $\kappa$ such that 
\begin{align}\label{eq:kappa}
   \frac{1}{\alpha-\beta} < \kappa < \frac{1}{\alpha-\beta} + \frac{1}{1/p_0-1/p_1} 
   = \frac{1}{\alpha-\beta} \left( 1 + \frac{\alpha-\beta}{1/p_0-1/p_1} \right).
\end{align}
Finally, we let
$$
    m_{M,\mu} :=\ceil{ C \, 2^{\kappa M+\left(1/[\alpha-\beta]-\kappa\right)\mu}},  \qquad \mu =M+1, \ldots, N_M,
$$
where $C>0$ is chosen such that (with $\nabla_\mu=\nabla_\mu(\alpha,\beta)$ as defined above) there holds
$$
    m_{M,\mu} \leq C \, 2^{\kappa M+\left(1/[\alpha-\beta]-\kappa\right)\mu} + 1 < C \, 2^{\mu/(\alpha-\beta)} + 1
    \leq \abs{\nabla_{\mu}} \sim 2^{\mu/(\alpha-\beta)},
$$
see \autoref{rem:counting}. 
Then the lower bound on $\kappa$ implies
$$
    \sum_{\mu=M+1}^{N_M}m_{M,\mu} 
    \lesssim N_M + 2^{\kappa M} \sum_{\mu=M+1}^{N_M} \, 2^{\left(1/(\alpha-\beta)-\kappa\right)\mu}
    \sim N_M + 2^{\kappa M}\, 2^{(1/(\alpha-\beta)-\kappa)M}
    \lesssim 2^{M/(\alpha-\beta)} 
$$
which together with \autoref{prop:linear_upperbound} proves that $\mathcal{B}_M$ as defined in \eqref{eq:BM} uses not more than $m=\ceil{c_1'\, 2^{M/(\alpha-\beta)}}$ coefficients of $\ba$. 

We are left with bounding the error 

\begin{align}\label{eq:ErrorUpper}
   \ba-\mathcal{B}_M(\ba)
   = \sum_{\bj \in (\Delta_{N_M})^C} \sum_{\bk \in\mathfrak{D}_{\bj}} a_{\bj,\bk}\,\be^{\bj,\bk}
    + \sum_{\mu=M+1}^{N_M} \sum_{(\bj,\bk)\in\nabla_{\mu}\setminus\Lambda_{M,\mu}}\!\! a_{\bj,\bk}\,\be^{\bj,\bk},
\end{align}
where $\be^{\bj,\bk}$ again denote the respective unit vectors. 
For the tail, i.e.\ the first sum in \eqref{eq:ErrorUpper}, denoted by $E_1(\ba)$, we can employ \eqref{eq:linear_error} with $M$ replaced by $N_M$ 
to conclude
$$
    \norm{E_1(\ba)\sep h_{p_1,q_1}^{r_1,s_1}b(\nabla)}
    \stackrel{\eqref{eq:linear_error}}{\lesssim} 2^{-N_M} \norm{\ba \sep h_{p_0,\infty}^{r_0,s_0}b(\nabla)}
    \sim \left(2^{M/(\alpha-\beta)}\right)^{-[(r_1-r_0)-(s_1-s_0)]} \norm{\ba \sep h_{p_0,\infty}^{r_0,s_0}b(\nabla)}.
$$

Hence, it remains to bound the second sum in \eqref{eq:ErrorUpper} which we call $E_2(\ba)$. 
To do so, let us rewrite the norm in the target space for arbitrary sequences $\bc=(c_{\bj,\bk})_{(\bj,\bk)\in\nabla} \in h_{p_1,q_1}^{r_1,s_1}b(\nabla)$ as follows. For the ease of presentation, w.l.o.g.\ we once more assume that $p_1,q_1<\infty$.
\begin{align*}
    &\norm{ \bc \sep h_{p_1,q_1}^{r_1,s_1}b(\nabla)}^{q_1} \\
    &\quad=\sum_{\mu=0}^{\infty} \sum_{\bj \in \mathfrak{L}_\mu} 2^{q_1 \!\big[(r_1-1/p_1) \abs{\bj}_1+ s_1\abs{\bj}_{\infty}\big]} \left[\sum_{\bk\in \mathfrak{D}_{\bj}} \abs{c_{\bj,\bk}}^{p_1} \right]^{q_1/p_1} \\
    &\quad=\sum_{\mu=0}^{\infty} \sum_{\bj \in \mathfrak{L}_\mu} 2^{-q_1 \!\big[ \big((r_0-r_1)-(1/p_0-1/p_1)-\epsilon\big) \abs{\bj}_1 - (s_1-s_0-\epsilon) \abs{\bj}_{\infty}\big]} \, 2^{-q_1(\abs{\bj}_1-\abs{\bj}_\infty)\epsilon} \\
        &\quad\qquad\qquad\qquad \cdot 2^{q_1\!\big[(r_0-1/p_0) \abs{\bj}_1 + s_0 \abs{\bj}_{\infty} \big]} \left[\sum_{\bk\in \mathfrak{D}_{\bj}} \abs{c_{\bj,\bk}}^{p_1} \right]^{q_1/p_1}  \\
    &\quad\sim \sum_{\mu=0}^{\infty} 2^{-q_1\mu} \sum_{\bj \in \mathfrak{L}_\mu} 2^{-q_1(\abs{\bj}_1-\abs{\bj}_\infty)\epsilon/2} \left[\sum_{\bk\in \mathfrak{D}_{\bj}} \Big( 2^{-(\abs{\bj}_1-\abs{\bj}_\infty)\epsilon/2} \, 2^{(r_0-1/p_0) \abs{\bj}_1+s_0\abs{\bj}_{\infty}} \abs{c_{\bj,\bk}} \Big)^{p_1} \right]^{q_1/p_1}. 
\end{align*}
Setting $\bc:=E_2(\ba)$, we can reduce the first sum to $\mu=M+1,\ldots,N_M$ and use Stechkin's \autoref{lem:stechkin} in order to bound the most inner sum for every fixed $\bj\in \mathfrak{L}_\mu$ by
\begin{align*}
    &\left[\sum_{\bk\in \mathfrak{D}_{\bj}} \Big( 2^{-(\abs{\bj}_1-\abs{\bj}_\infty)\epsilon/2} \, 2^{(r_0-1/p_0) \abs{\bj}_1+s_0 \abs{\bj}_{\infty}} \abs{c_{\bj,\bk}} \Big)^{p_1} \right]^{q_1/p_1} \\
    &\quad \leq \left[\sum_{(\overline{\bj},\bk)\in \nabla_\mu\setminus \Lambda_{M,\mu} } \left( 2^{-\big(\abs{\overline{\bj}}_1-\abs{\overline{\bj}}_\infty\big)\epsilon/2} \, 2^{(r_0-1/p_0) \abs{\overline{\bj}}_1+s_0 \abs{\overline{\bj}}_{\infty}} \abs{a_{\overline{\bj},\bk}} \right)^{p_1} \right]^{q_1/p_1} \\
    &\quad \leq (\abs{\Lambda_{M,\mu}}+1)^{-q_1(1/p_0-1/p_1)} \left[\sum_{(\overline{\bj},\bk)\in \nabla_\mu} \left( 2^{-\big(\abs{\overline{\bj}}_1-\abs{\overline{\bj}}_\infty\big)\epsilon/2} \, 2^{(r_0-1/p_0) \abs{\overline{\bj}}_1+s_0 \abs{\overline{\bj}}_{\infty}}  \abs{a_{\overline{\bj},\bk}} \right)^{p_0} \right]^{q_1/p_0} \\
    &\quad \leq m_{M,\mu}^{-q_1(1/p_0-1/p_1)} \left[\sum_{\overline{\bj}\in \mathfrak{L}_\mu} 2^{-p_0\big(\abs{\overline{\bj}}_1-\abs{\overline{\bj}}_\infty\big)\epsilon/2} \left( 2^{(r_0-1/p_0)\abs{\overline{\bj}}_1+s_0 \abs{\overline{\bj}}_{\infty}} \!\left[ \sum_{\bk \in \mathfrak{D}_{\overline{\bj}}} \abs{a_{\overline{\bj},\bk}}^{p_0} \right]^{1/p_0}\right)^{p_0} \right]^{q_1/p_0}\!.
\end{align*}
Now \autoref{lem:log-killer} yields
\begin{align*}
    &\left[\sum_{\bk\in \mathfrak{D}_{\bj}} \Big( 2^{-(\abs{\bj}_1-\abs{\bj}_\infty)\epsilon/2} \, 2^{(r_0-1/p_0) \abs{\bj}_1+s_0\abs{\bj}_{\infty}} \abs{c_{\bj,\bk}} \Big)^{p_1} \right]^{q_1/p_1} \\
    &\qquad \leq m_{M,\mu}^{-q_1(1/p_0-1/p_1)} \left[\sum_{\overline{\bj}\in \mathfrak{L}_\mu} 2^{-p_0\big(\abs{\overline{\bj}}_1-\abs{\overline{\bj}}_\infty\big)\epsilon/2} \right]^{q_1/p_0} \norm{\ba \sep h^{r_0,s_0}_{p_0,\infty}b(\nabla)}^{q_1} \\
    &\qquad \lesssim m_{M,\mu}^{-q_1(1/p_0-1/p_1)} \norm{\ba \sep h^{r_0,s_0}_{p_0,\infty}b(\nabla)}^{q_1}, \qquad \bj\in \mathfrak{L}_\mu.
\end{align*}
Therefore, we can conclude
\begin{align*}
    \norm{ E_2(\ba) \sep h_{p_1,q_1}^{r_1,s_1}b(\nabla)}^{q_1}
    &\lesssim \sum_{\mu=M+1}^{N_M} 2^{-q_1\mu} \sum_{\bj \in \mathfrak{L}_\mu} 2^{-q_1(\abs{\bj}_1-\abs{\bj}_\infty)\epsilon/2} \, m_{M,\mu}^{-q_1(1/p_0-1/p_1)} \norm{\ba \sep h^{r_0,s_0}_{p_0,\infty}b(\nabla)}^{q_1}  \\ 
    &\lesssim \norm{\ba \sep h^{r_0,s_0}_{p_0,\infty}b(\nabla)}^{q_1}  \sum_{\mu=M+1}^{N_M} 2^{-q_1\mu} \, m_{M,\mu}^{-q_1(1/p_0-1/p_1)},
\end{align*}
where we once again used \autoref{lem:log-killer}.
Finally, in view of the definition of $m_{M,\mu}$ and the upper bound on $\kappa$ from  \eqref{eq:kappa}, we see that the remaining sum can be estimated by
\begin{align*}
    \sum_{\mu= M+1}^{N_M} 2^{-q_1\mu}\, m_{M,\mu}^{-q_1(1/p_0-1/p_1)}
    &\sim \sum_{\mu= M+1}^{N_M} 2^{-q_1\mu + [\kappa M+(1/(\alpha-\beta)-\kappa)\mu] \,[-q_1(1/p_0-1/p_1)]} \\
    &= 2^{-q_1\kappa M(1/p_0-1/p_1)} \sum_{\mu= M+1}^{N_M} 2^{-q_1 \mu \big(1 + [1/(\alpha-\beta)-\kappa] \,(1/p_0-1/p_1) \big)} \\
    &\sim 2^{-q_1\kappa M(1/p_0-1/p_1)} \, 2^{-q_1 M \big(1 + [1/(\alpha-\beta)-\kappa] \,(1/p_0-1/p_1) \big)} \\
    &= 2^{-q_1 M \big(1 + [1/(\alpha-\beta)] \,(1/p_0-1/p_1) \big)} \\
    &= \left[ \left(2^{M/(\alpha-\beta)}\right)^{-[(r_0-r_1)-(s_0-s_1)]} \right]^{q_1} \\
    &\sim m^{-[(r_0-r_1)-(s_0-s_1)]q_1}.
\end{align*}
In conclusion, we have shown that
\begin{align*}
    \norm{E_2(\ba)\sep h_{p_1,q_1}^{r_1,s_1}b(\nabla)}
    &\lesssim m^{-[(r_0-r_1)-(s_0-s_1)]} \norm{\ba\sep h_{p_0,\infty}^{r_0,s_0}b(\nabla)}
\end{align*}
and the proof is complete.
\end{proof}

\begin{remark}
    The presentation of the proof above is a matter of taste. 
    According to \autoref{rem:widths}(iv) it would also be possible to apply the (pre)additivity of   $\sigma_m$ stated in \eqref{eq:additivity} in connection with Maiorov's discretization technique \cite{Mai75} on the level of pseudo-$s$-numbers, in order to conclude the proof with a more abstract presentation. 
    In the context of Weyl and Bernstein numbers, this approach has been used, e.g., in \cite{VKN16}.
\end{remark}

Similar to \autoref{cor:linear_approx} we derive the following \autoref{cor:nonlin_approx} which is optimal in view of \autoref{prop:lowerbound}.
\begin{corollary}\label{cor:nonlin_approx}
    For $d\in\N$ let $x,y\in\{b,f\}$,  $0< p_0, p_1, q_0, q_1 \leq \infty$, and $r_0,r_1,s_0,s_1 \in\R$ s.t.
    $$
        r_0-r_1 - \left( \frac{1}{p_0} - \frac{1}{p_1} \right)_+ > s_1-s_0 > 0
    $$
    (with $p_0<\infty$ if $x=f$ and $p_1<\infty$ if $y=f$, respectively). 
    Then 
    $$
        \sigma_m \!\left( \id \colon h^{r_0,s_0}_{p_0,q_0}x(\nabla) \to h^{r_1,s_1}_{p_1,q_1}y(\nabla); E \right) 
        \lesssim m^{-[(r_0-r_1)-(s_1-s_0)]}, \qquad m\geq m_0.
    $$
    
\end{corollary}
\begin{proof}
    If $p_0\geq p_1$, the assertion follows from \autoref{rem:linear_optimal}(iii).  
    If otherwise $p_0<p_1$, we employ \autoref{prop:nonlin_upperbound} together with \autoref{lem:embeddings} and the monotonicity of~$\sigma_m(\id;E)$.
\end{proof}

\section{Approximation rates in function spaces}\label{sect:fkt_rates}
We are now able to formulate our main result, i.e., transfer the assertions from \autoref{sec:seqapprox} to the level of function spaces of hybrid smoothness.

\begin{theorem}\label{thm:MainRes} 
    For $d\in\N$ let $\Omega\subset\R^d$ be a bounded domain.
    Let $X,Y\in\{B,F\}$ and $0<p_0,p_1,q_0,q_1\leq \infty$ (with $p_0<\infty$ if $X=F$ and $p_1<\infty$ if $Y=F$, respectively), as well as $r_0,r_1,s_0,s_1\in\R$ such that 
    \begin{align*}
        r_0-r_1 - \left( \frac{1}{p_0} - \frac{1}{p_1} \right)_{+} > s_1-s_0 > 0.
    \end{align*}
    Then the embedding $\Id \colon H_{p_0,q_0}^{r_0,s_0}X(\Omega)\to H_{p_1,q_1}^{r_1,s_1}Y(\Omega)$ is compact and for some $M_0,c,c'>0$ there exist sequences of algorithms $(\B_M)_{M\in\N}$ and $(\A_M)_{M\in\N}$ such that
    \begin{enumerate}[label=(\roman*.)]
        \item for all $M\geq M_0$ and $m:=m(M):=\ceil{c\,2^{M/[(r_0-r_1)-(s_1-s_0)-(1/p_0-1/p_1)_+]}}$ there holds 
        \begin{align*}
            \sigma_m \big(\Id; \Psi \big)
            &\sim \sup_{\norm{f \sep H_{p_0,q_0}^{r_0,s_0}X(\Omega)}\leq 1}\norm{f-\mathcal{B}_M(f) \sep H_{p_1,q_1}^{r_1,s_1}Y(\Omega)} 
            \sim \, m^{-[(r_0-r_1)-(s_1-s_0)]}.
        \end{align*}
        For each input, $\mathcal{B}_M$ produces a linear combination of at most $m$ adaptively chosen elements from the (hyperbolic wavelet) dictionary $\Psi:=\left\{\psi^\lambda\vert_\Omega \sep \lambda\in\nabla\right\}$.

        \item for all $M\geq M_0$ and $m:=m(M):=\ceil{c'\,2^{M/[(r_0-r_1)-(s_1-s_0)-(1/p_0-1/p_1)_+]}}$ we have
        \begin{align*}
            \zeta_m \big(\Id; \Psi \big)
            \sim d_m \big(\Id; \Psi \big)
            &\sim \sup_{\norm{f \sep H_{p_0,q_0}^{r_0,s_0}X(\Omega)}\leq 1} \norm{f-\mathcal{A}_M(f) \sep H_{p_1,q_1}^{r_1,s_1}Y(\Omega)} \\
             &\sim \, m^{-[(r_0-r_1)-(s_1-s_0)-(1/p_0-1/p_1)_+]}.
        \end{align*}
        For each input, $\mathcal{A}_M$ uses a linear combination of the same at most $m$ elements from $\Psi$.
\end{enumerate}
\end{theorem} 
\begin{proof}
In \autoref{prop:lifting} it was shown that $\sigma_m(\Id;\Psi)\sim \sigma_m(\id; E)$, where $E$ is the set of unit vectors at the level of sequence spaces and $\id\colon h_{p_0,q_0}^{r_0,s_0}x(\nabla)\to h_{p_1,q_1}^{r_1,s_1}y(\nabla)$ denotes the discrete analogue of $\Id$. 
Inserting the the upper bound of $\sigma_m(\id;E)$ from \autoref{cor:nonlin_approx} and the lower bound from \autoref{prop:lowerbound} proves \emph{(i.)}.
Assertion \emph{(ii.)} is shown likewise using \autoref{cor:linear_approx} and \autoref{prop:lowerbound_linear}, respectively. 
Finally, compactness of $\Id$ follows from \autoref{rem:widths}(iii).
\end{proof}

Note that this assertion combined with \autoref{prop:conntoclassspaces} especially proves \autoref{thm:mainsoboloev}. The remarks given there apply likewise for the more general situation of \autoref{thm:MainRes}.

\begin{appendices}\section{}\label{sect:Appendix}
\subsection{Proof of \autoref{prop:lifting}}\label{subsec:Proofs}

The proof of \autoref{prop:lifting} given below is based on the so-called lattice property of our sequence spaces and a carefully chosen non-linear extension operator $\mathcal{E}^*$ from $\Omega$ to $\R^d$ which is inspired by ideas in \cite[Section 4.6.6]{HanSic2011}.  

\begin{lemma}[Lattice property]
\label{lem:lattice_prop}
Let $y\in\{b,f\}$, $0<p,q\leq \infty$ (with $p<\infty$ if $y=f$) and $r,s\in\R$. Then for all $\Lambda\subset\nabla$
$$
	\inf_{\substack{c_\lambda \in\C,\\ \lambda \in \Lambda}} \bigg\|\ba- \sum_{\lambda\in\Lambda} c_\lambda\, \be^\lambda \,\bigg|\, h^{r,s}_{p,q}y(\nabla) \bigg\|
	= \bigg\| \ba- \sum_{\lambda\in\Lambda} a_\lambda\, \be^\lambda \,\bigg|\, h^{r,s}_{p,q}y(\nabla) \bigg\|, 
	\qquad \ba = \big(a_\lambda\big)_{\lambda\in\nabla}\in h^{r,s}_{p,q}y(\nabla).
$$
\end{lemma}
\begin{proof}
The non-trivial estimate follows from the lattice structure of the spaces $\mathcal{s}:=h^{r,s}_{p,q}y(\nabla)$:
$$
	\bigg\| \ba- \sum_{\lambda\in\Lambda} c_\lambda\, \be^\lambda \,\bigg|\, \mathcal{s} \bigg\|
	= \bigg\| \bigg( \ba- \sum_{\lambda\in\Lambda} a_\lambda\, \be^\lambda \bigg) + \bigg( \sum_{\lambda\in\Lambda} (a_\lambda-c_\lambda)\, \be^\lambda \bigg) \,\bigg|\, \mathcal{s} \bigg\|
	\geq \bigg\| \ba- \sum_{\lambda\in\Lambda} a_\lambda\, \be^\lambda \,\bigg|\, \mathcal{s} \bigg\|
$$
for every choice of $c_\lambda$, $\lambda\in\Lambda$.
\end{proof}

Let us first bound the quantities of interest on the level of function spaces by corresponding ones for hybrid sequence spaces.

\begin{lemma}\label{lem:Id_smaller_id}
    For $d\in\N$ let $x,y\in\{b,f\}$ and $0<p_0,p_1,q_0,q_1\leq \infty$ (with $p_0<\infty$ if $x=f$ and $p_1<\infty$ if $y=f$), as well as $r_0,r_1,s_0,s_1\in\R$ be such that 
    $\id \colon h_{p_0,q_0}^{r_0,s_0}x(\nabla) \hookrightarrow h_{p_1,q_1}^{r_1,s_1}y(\nabla)$. 
    Then $\Id \colon H_{p_0,q_0}^{r_0,s_0}X(\Omega) \hookrightarrow H_{p_1,q_1}^{r_1,s_1}Y(\Omega)$ and
    $$
        \zeta_m(\Id;\Psi) \lesssim d_m(\id;E)
        \quad\text{as well as}\quad
        \sigma_m(\Id;\Psi) \lesssim \sigma_m(\id;E), \qquad m\in\N_0.
    $$
\end{lemma}
\begin{proof}
    For every $f\in H_{p,q}^{r,s}X(\Omega)$ there exists an extension
    $
        F = \sum_{\lambda\in \N_{0}^d \times \Z^d} a_{\lambda} \, \psi^{\lambda} \in H_{p,q}^{r,s}X(\R^d)
    $
    with $\frac{1}{2}\norm{F\sep H_{p,q}^{r,s}X(\R^d)} \leq \norm{f\sep H_{p,q}^{r,s}X(\Omega)}$. 
    Setting $\ba_f := ( a_{\lambda} )_{\lambda \in\nabla}$, we can thus define
    another local, but possibly non-linear extension to $f$,
    \begin{align}\label{eq:Estern}
        \mathcal{E}^*(f) := \sum_{\lambda\in \nabla} a_{\lambda} \,\psi^\lambda 
        \;\in\; H_{p,q}^{r,s}X(\R^d),
    \end{align}
    such that with constants independent of $f$ there holds
    \begin{align}\label{eq:Estern2}
        \norm{f\sep H_{p,q}^{r,s}X(\Omega)} 
        \gtrsim \norm{\mathcal{E}^*(f) \sep H_{p,q}^{r,s}X(\R^d)} 
        = \norm{ \ba_f \sep h_{p,q}^{r,s}x(\nabla)}.
    \end{align}
    If $\id\in\mathcal{L}\big( h_{p_0,q_0}^{r_0,s_0}x(\nabla), h_{p_1,q_1}^{r_1,s_1}y(\nabla)\big)$ and $f\in H_{p_0,q_0}^{r_0,s_0}X(\Omega)$, then $\mathcal{E}^*(f)\in\S'(\R^d)$ provides an extension to it for which
    $$
        \norm{\mathcal{E}^*(f) \sep H_{p_1,q_1}^{r_1,s_1}Y(\R^d)} 
        =\norm{ \ba_f \sep h_{p_1,q_1}^{r_1,s_1}y(\nabla)}
        \lesssim \norm{ \ba_f \sep h_{p_0,q_0}^{r_0,s_0}x(\nabla)}
        \lesssim \norm{f\sep H_{p_0,q_0}^{r_0,s_0}X(\Omega)}
    $$
    is finite. Hence, $f\in H_{p_1,q_1}^{r_1,s_1}Y(\Omega)$ and $\norm{f\sep H_{p_1,q_1}^{r_1,s_1}Y(\Omega)} \lesssim \norm{f\sep H_{p_0,q_0}^{r_0,s_0}X(\Omega)}$.
    
    Now let $\Lambda_m\subset\nabla$ with $\abs{\Lambda_m}\leq m$ be arbitrarily fixed. 
    Further let us choose continuous linear functionals $c_\lambda^* \in H_{p,q}^{r,s}X(\R^d)'$ such that $c_\lambda^*(\psi^{\rho})=\delta_{\lambda,\rho}$ for all $\lambda,\rho\in\nabla$ and define $\widetilde{c}_\lambda:=c_\lambda^* \circ \mathcal{E}^* \colon H_{p,q}^{r,s}X(\Omega) \to \C$.
    Then for all $f\in H_{p_0,q_0}^{r_0,s_0}X(\Omega)$ there holds
    $$
        \widetilde{c}_\lambda(f) = c^*_\lambda \bigg( \sum_{\rho\in \nabla} a_{\rho} \,\psi^\rho \bigg) = a_\lambda, \qquad \lambda\in\nabla,
    $$
    with $\ba_f = ( a_{\lambda} )_{\lambda \in\nabla}\in h_{p_0,q_0}^{r_0,s_0}x(\nabla)$ as in \eqref{eq:Estern} and hence
    \begin{align}
        \bigg\| f - \sum_{\lambda\in\Lambda_m} \widetilde{c}_\lambda(f) \, \psi^\lambda\vert_{\Omega} \,\bigg|\,  H^{r_1,s_1}_{p_1,q_1}Y(\Omega) \bigg\|
        &\leq \bigg\| \mathcal{E}^*(f) - \sum_{\lambda\in\Lambda_m} \widetilde{c}_\lambda(f) \, \psi^\lambda  \,\bigg|\,  H^{r_1,s_1}_{p_1,q_1}Y(\R^d) \bigg\| \nonumber\\
        &= \bigg\| \ba_f - \sum_{\lambda\in\Lambda_m} a_\lambda \, \be^\lambda  \,\bigg|\,  h^{r_1,s_1}_{p_1,q_1}y(\nabla) \bigg\| \nonumber\\
        &= \inf_{\substack{c_\lambda \in \C,\\ \lambda\in\Lambda_m}} \bigg\| \ba_f - \sum_{\lambda\in\Lambda_m} c_\lambda \, \be^\lambda  \,\bigg|\,  h^{r_1,s_1}_{p_1,q_1}y(\nabla) \bigg\|, \label{eq:proof_zeta}
    \end{align}
    due to \autoref{lem:lattice_prop}. 
    So, \eqref{eq:Estern2} implies
    \begin{align*}
        &\sup_{\norm{f \sep H_{p_0,q_0}^{r_0,s_0}X(\Omega)}\leq 1} \bigg\| f - \sum_{\lambda\in\Lambda_m} \widetilde{c}_\lambda(f) \, \psi^\lambda\vert_{\Omega} \,\bigg|\,  H^{r_1,s_1}_{p_1,q_1}Y(\Omega) \bigg\| \\
        &\qquad \leq \sup_{\norm{f \sep H_{p_0,q_0}^{r_0,s_0}X(\Omega)}\leq 1} \inf_{\substack{c_\lambda \in \C,\\ \lambda\in\Lambda_m}} \bigg\| \ba_f - \sum_{\lambda\in\Lambda_m} c_\lambda \, \be^\lambda  \,\bigg|\,  h^{r_1,s_1}_{p_1,q_1}y(\nabla) \bigg\| \\
        &\qquad \lesssim \sup_{\norm{\ba \sep h_{p_0,q_0}^{r_0,s_0}x(\nabla)}\leq 1} \inf_{\substack{c_\lambda \in \C,\\ \lambda\in\Lambda_m}} \bigg\| \ba - \sum_{\lambda\in\Lambda_m} c_\lambda \, \be^\lambda  \,\bigg|\,  h^{r_1,s_1}_{p_1,q_1}y(\nabla) \bigg\|
    \end{align*}
    which (by taking the infimum w.r.t.\ $\Lambda_m$) yields $\zeta_m(\Id;\Psi) \lesssim d_m(\id;E)$.
    
    For the best $m$-term widths we can argue similarly. Indeed, \eqref{eq:proof_zeta} implies
    $$
        \inf_{\substack{\Lambda_m\subset\nabla,\\ \abs{\Lambda_m}\leq m}} \inf_{\substack{c_\lambda \in \C,\\ \lambda\in\Lambda_m}} \bigg\| f - \sum_{\lambda\in\Lambda_m} c_\lambda \, \psi^\lambda\vert_{\Omega} \,\bigg|\,  H^{r_1,s_1}_{p_1,q_1}Y(\Omega) \bigg\|
        \leq \inf_{\substack{\Lambda_m\subset\nabla,\\ \abs{\Lambda_m}\leq m}} \inf_{\substack{c_\lambda \in \C,\\ \lambda\in\Lambda_m}} \bigg\| \ba_f - \sum_{\lambda\in\Lambda_m} c_\lambda \, \be^\lambda  \,\bigg|\,  h^{r_1,s_1}_{p_1,q_1}y(\nabla) \bigg\|
    $$
    for all $f\in H_{p_0,q_0}^{r_0,s_0}X(\Omega)$, hence  taking the sup w.r.t.\ $\norm{f \sep H_{p_0,q_0}^{r_0,s_0}X(\Omega)}\leq 1$ proves the claim.
\end{proof}

The converse to \autoref{lem:Id_smaller_id} reads as follows:

\begin{lemma}\label{lem:id_smaller_Id}
    For $d\in\N$ let $X,Y\in\{B,F\}$ and $0<p_0,p_1,q_0,q_1\leq \infty$ (with $p_0<\infty$ if $X=F$ and $p_1<\infty$ if $Y=F$, respectively), as well as $r_0,r_1,s_0,s_1\in\R$ be such that $\Id \colon H_{p_0,q_0}^{r_0,s_0}X(\Omega) \hookrightarrow H_{p_1,q_1}^{r_1,s_1}Y(\Omega)$.
    Then $\id \colon h_{p_0,q_0}^{r_0,s_0}x(\nabla) \hookrightarrow h_{p_1,q_1}^{r_1,s_1}y(\nabla)$ and
    $$
        d_m(\id;E) \lesssim d_m(\Id;\Psi) 
        \quad\text{as well as} \quad
        \sigma_m(\id;E) \lesssim \sigma_m(\Id;\Psi), \qquad m\in\N_0.
    $$
\end{lemma}
\begin{proof}
    Set $\nabla':=\{\lambda\in\nabla \sep \supp(\psi^\lambda)\cap \partial\Omega = \emptyset\}$. Then every $\ba=(a_\lambda)_{\lambda\in\nabla'} \in h^{r,s}_{p,q}x(\nabla')$ defines a distribution $f_{\ba} := \sum_{\lambda\in\nabla'} a_\lambda \, \psi^{\lambda}$ in $H^{r,s}_{p,q}X(\R^d)$ with $\supp(f_{\ba})\subset \Omega$. 
    Hence, we actually have $f_{\ba} \in H^{r,s}_{p,q}X(\Omega)$ and
    $$
        \norm{f_{\ba} \sep H^{r,s}_{p,q}X(\Omega)} 
        = \norm{f_{\ba} \sep H^{r,s}_{p,q}X(\R^d)}
        = \norm{\ba \sep h^{r,s}_{p,q}x(\nabla')}.
    $$
    
    Therefore, $\Id\in\mathcal{L}\big( H_{p_0,q_0}^{r_0,s_0}X(\Omega), H_{p_1,q_1}^{r_1,s_1}Y(\Omega)\big)$ and $\ba \in h_{p_0,q_0}^{r_0,s_0}x(\nabla')$ yield that
    $$
        \norm{\ba \sep h_{p_1,q_1}^{r_1,s_1}y(\nabla')} 
        = \norm{f_{\ba} \sep H^{r_1,s_1}_{p_1,q_1}Y(\Omega)}
        \lesssim \norm{f_{\ba} \sep H^{r,s}_{p,q}X(\Omega)}
        = \norm{\ba \sep h_{p_0,q_0}^{r_0,s_0}x(\nabla')}
    $$
    is finite and thus $h_{p_0,q_0}^{r_0,s_0}x(\nabla') \hookrightarrow h_{p_1,q_1}^{r_1,s_1}y(\nabla')$.
    
    Now let $\Lambda_m\subset \nabla'$ with $\abs{\lambda_m}\leq m$ be arbitrarily fixed. Then for all $\ba \in h_{p_0,q_0}^{r_0,s_0}x(\nabla')$ we can select coefficients $\widehat{c}_\lambda$, $\lambda\in\Lambda_m$, with
    \begin{align*}
        \inf_{\substack{c_\lambda \in \C,\\ \lambda\in\Lambda_m}} \bigg\| \ba - \sum_{\lambda\in\Lambda_m} c_\lambda \, \be^\lambda \,\bigg|\,  h^{r_1,s_1}_{p_1,q_1}y(\nabla') \bigg\|
        &\leq \bigg\| \ba - \sum_{\lambda\in\Lambda_m} \widehat{c}_\lambda \, \be^\lambda \,\bigg|\,  h^{r_1,s_1}_{p_1,q_1}y(\nabla') \bigg\| \\
        &= \bigg\| f_{\ba} - \sum_{\lambda\in\Lambda_m} \widehat{c}_\lambda \, \psi^\lambda \,\bigg|\,  H^{r_1,s_1}_{p_1,q_1}Y(\Omega) \bigg\| \\
        &\leq 2\inf_{\substack{c_\lambda \in \C,\\ \lambda\in\Lambda_m}} \bigg\| f_{\ba} - \sum_{\lambda\in\Lambda_m} c_\lambda \, \psi^\lambda\vert_{\Omega} \,\bigg|\,  H^{r_1,s_1}_{p_1,q_1}Y(\Omega) \bigg\|.
    \end{align*}
    Next, we take the sup over all $\norm{\ba \sep h^{r_0,s_0}_{p_0,q_0}x(\nabla')}\leq 1$ which at the right-hand side can be replaced by the sup over all $f_{\ba} \in H^{r_0,s_0}_{p_0,q_0}X(\Omega)$ with (quasi-)norm at most one.
    So, $d_m(\id'\colon h^{r_0,s_0}_{p_0,q_0}x(\nabla') \to h^{r_1,s_1}_{p_1,q_1}y(\nabla');E\vert_{\nabla'})\lesssim d_m(\Id;\Psi)$ and essentially the same arguments show that also $\sigma_m(\id'\colon h^{r_0,s_0}_{p_0,q_0}x(\nabla') \to h^{r_1,s_1}_{p_1,q_1}y(\nabla');E\vert_{\nabla'})\lesssim \sigma_m(\Id;\Psi)$.
    
    Finally note that, since $\Omega$ is bounded and contains the unit cube, the number of translates~$\bk$ in $\mathfrak{D}_{\bj}$ and $\mathfrak{D}_{\bj}'$ (related to $\nabla$ and~$\nabla'$, respectively) both scale like $2^{\abs{\bj}_1}$. 
    Therefore, in all proven assertions the sequence spaces on $\nabla'$ can be replaced by corresponding ones on $\nabla$ and thus the proof is complete.
\end{proof}

Now we are well-prepared to prove the lifting assertion stated in \autoref{prop:lifting}.

\begin{proof}[Proof of {\autoref{prop:lifting}}]
    Combining \autoref{lem:Id_smaller_id} and \ref{lem:id_smaller_Id} shows the continuity statement, as well as the assertion on best-$m$-term widths.
    Moreover, together with \eqref{eq:width_ordering} they yield
    $$
        d_m(\id;E) \lesssim d_m(\Id;\Psi) \leq \zeta_m(\Id;\Psi) \lesssim d_m(\id;E)
        \qquad\text{and}\qquad 
        d_m(\id;E) \leq \zeta_m(\id;E).
    $$
    To finish the proof, we note that \autoref{lem:lattice_prop} implies that for all $m\in\N_0$
    \begin{align*}
        d_m(\id;E) 
        &= \inf_{\substack{\Lambda_m \subset \nabla,\\ \abs{\Lambda_m}\leq m}} \sup_{\norm{\ba \sep h^{r_0,s_0}_{p_0,q_0}x(\nabla)}\leq 1} \bigg\| \ba - \sum_{\lambda\in\Lambda_m} a_\lambda \, \be^\lambda \,\bigg|\,  h^{r_1,s_1}_{p_1,q_1}y(\nabla) \bigg\|
        \geq \zeta_m(\id;E).\qedhere
    \end{align*}
\end{proof}

\subsection{Sparse approximation of sequences: Stechkin's inequality}
Our upper bounds are based on a result which is frequently attributed to Sergey Stechkin. For the convenience of the reader, we add its simple proof based on \cite[Lemma~3.3]{KreTob11}.
\begin{lemma}[Stechkin]\label{lem:stechkin}
    Let $\mathcal{I}\neq\emptyset$ denote some countable index set, $0<p_0 \leq p_1 \leq \infty$, and $\ba=(a_i)_{i\in \mathcal{I}}\in \ell_{p_0}(\mathcal{I})$ be some real or complex sequence. 
    Then for all finite subsets $\Lambda \subseteq \mathcal{I}$ with $\abs{a_\lambda}\geq \abs{a_i}$ for all $\lambda\in\Lambda$ and $i\in \mathcal{I}\setminus \Lambda$ there holds
    $$
        \left( \sum_{i\in \mathcal{I}\setminus\Lambda} \abs{a_i}^{p_1} \right)^{1/p_1} 
        \leq (\abs{\Lambda}+1)^{-(1/p_0-1/p_1)} \left( \sum_{i\in \mathcal{I}} \abs{a_i}^{p_0} \right)^{1/p_0}
    $$
    with the usual modifications if $p_1$ or $p_0$ equal infinity. 
\end{lemma}
Note that \autoref{lem:stechkin} implies
$$
        \sigma_m(\id\colon \ell_{p_0}(\mathcal{I}) \to \ell_{p_1}(\mathcal{I}); E) \lesssim m^{-(1/p_0-1/p_1)}
$$
with $E$ denoting the unit vectors in the corresponding sequence spaces.

\begin{proof}
The cases $\Lambda=\emptyset$, or $\Lambda=\mathcal{I}$, or $p_0=p_1$ are trivial. 
Hence, we can assume that $\abs{\mathcal{I}}>\abs{\Lambda}\geq 1$ and $p_0<p_1\leq\infty$. 
Let $(b_n)_{n=1}^{\abs{\mathcal{I}}}\subset\R$ be any non-increasing rearrangement of $(\abs{a_i})_{i\in \mathcal{I}}$, i.e., $b_n\geq b_{n+1}$ for all $n$. 
Then for $m:=\abs{\Lambda}+1$ there holds 
$$
    m\, b_m^{p_0} \leq b_1^{p_0} + \ldots + b_m^{p_0} \leq \sum_{n=1}^{\abs{\mathcal{I}}} b_n^{p_0}
    \quad \text{and hence}\quad
    b_m^{1-p_0/p_1} \leq m^{-(1/p_0-1/p_1)} \left[ \sum_{n=1}^{\abs{\mathcal{I}}} b_n^{p_0} \right]^{1/p_0-1/p_1},
$$
since $1/p_0>1/p_1$. If $p_1=\infty$, this implies the claim as follows:
$$
    \max_{i\in \mathcal{I}\setminus\Lambda } \abs{a_i} 
    = \max_{n\geq m} b_n 
    \leq b_{m}
    \leq m^{-1/p_0} \left[ \sum_{n=1}^{\abs{\mathcal{I}}} b_n^{p_0} \right]^{1/p_0} 
    = (\abs{\Lambda}+1)^{-1/p_0} \left[ \sum_{i\in \mathcal{I}} \abs{a_i}^{p_0} \right]^{1/p_0}.
$$
On the other hand, if $p_1<\infty$, we can argue similarly and obtain
\begin{align*}
    \left[ \sum_{i\in \mathcal{I}\setminus\Lambda} \abs{a_i}^{p_1} \right]^{1/p_1} 
    \!= \left[ \sum_{n=m}^{\abs{\mathcal{I}}} b_n^{p_1-p_0}b_n^{p_0} \right]^{1/p_1} 
    \!\leq b_m^{1-p_0/p_1} \left[ \sum_{n=m}^{\abs{\mathcal{I}}} b_n^{p_0} \right]^{1/p_1} 
    \!\leq m^{-(1/p_0-1/p_1)} \left[ \sum_{n=1}^{\abs{\mathcal{I}}} b_n^{p_0} \right]^{1/p_0}
\end{align*}
which finishes the proof.
\end{proof}

\subsection{Combinatorics}\label{sect:combinatorics}
In this appendix, we collect estimates related to the sets $\Delta_\mu:=\Delta_\mu(\alpha,\beta)$ and $\mathfrak{L}_\mu:=\mathfrak{L}_\mu(\alpha,\beta)$ introduced in \autoref{def:delta}. 
We start with bounding their cardinality.
\begin{lemma}\label{lem:countingS}
Let $d\in\N$.
\begin{enumerate}[label=(\roman*.)]
    \item If $\alpha,\beta\geq 0$, then
    $$
        \abs{\Delta_\mu} \gtrsim \mu^d, \qquad \mu\in\N.
    $$
    \item If $\alpha \geq 0$ and $\beta<\alpha$, then
    $$
        \abs{\Delta_\mu} \lesssim \mu^d, \qquad \mu \in\N.
    $$
    \item If $\alpha > \beta \geq 0$, we have
    $$
        \abs{\mathfrak{L}_\mu} \sim \mu^{d-1}, \qquad \mu \in\N \setminus\{1\}.
    $$
\end{enumerate}
\end{lemma}
\begin{proof}
The last statement follows from the previous ones since
$$
    \abs{\mathfrak{L}_{\mu+1}}=\abs{\Delta_{\mu+1}}-\abs{\Delta_{\mu}} \sim (\mu+1)^d - \mu^d 
    = \sum_{k=0}^{d-1} \binom{d}{k} 1^{d-k} \mu^k
    \sim \mu^{d-1} \sim (\mu+1)^{d-1}.
$$

For the lower bound we note that $\mathfrak{L}:=\{0,1,\ldots, \floor{\frac{\mu}{\alpha d}}\}^d \subset \Delta_\mu$, since every $\bj\in \mathfrak{L}$ satisfies
$$
    \mu \geq \alpha d \abs{\bj}_\infty 
    \geq \alpha \abs{\bj}_1 
    \geq \alpha \abs{\bj}_1 - \beta \abs{\bj}_\infty,
$$
i.e. $\bj\in \Delta_\mu$. So, 
\begin{align*}
    \abs{\Delta_\mu} \geq \abs{\mathfrak{L}} 
    &= \begin{cases}
    \infty, & \alpha=0,\\
    (1+\floor{\frac{\mu}{\alpha d}})^d \geq (\frac{\mu}{\alpha d})^d, &\alpha >0,
    \end{cases}\\
    &\gtrsim \mu^d.
\end{align*}

Similarly, every $\bj\in\Delta_\mu$ satisfies $\mu \geq \alpha \abs{\bj}_1 - \beta \abs{\bj}_\infty \geq (\alpha- \beta) \abs{\bj}_\infty$ such that $\bj$ belongs to $\mathfrak{U}:= \left\{ 0,1,\ldots, \floor{\frac{\mu}{\alpha-\beta}}\right\}^d$. 
Therefore, $\abs{\Delta_\mu}\leq \abs{\mathfrak{U}} \leq (1+ \frac{\mu}{\alpha-\beta})^d \leq ( \frac{2\mu}{\alpha-\beta})^d \lesssim \mu^d$ if $\mu\geq \alpha-\beta$. 
Otherwise, $\Delta_\mu \lesssim 1 \lesssim \mu^d$.
\end{proof}

Further, we shall use the following sharp estimate which generalizes \cite[Lemma~6.3]{ByrDunSic+16}.
\begin{lemma}\label{lem:counting}
Let $d\in \N$, $\alpha>\beta>0$, and $\delta>0$. Then
$$
    \sum_{\bj \in \Delta_\mu} 2^{\delta \abs{\bj}_1} \sim 2^{\delta \mu / (\alpha-\beta)}, \qquad \mu\geq \alpha-\beta.
$$
\end{lemma}
\begin{proof}
\emph{Step 1 (Lower bound). } We first show that
$$
    \bj^* := \floor{\frac{\mu}{\alpha-\beta}} \, (1,0,\ldots, 0)
$$
belongs to $\Delta_\mu$.
Indeed, 
\begin{align*}
    \bj^* \in \Delta_\mu
    \quad &\Longleftrightarrow \quad \alpha \abs{\bj^*}_1 - \beta \abs{\bj^*}_\infty \leq \mu
    \quad \Longleftrightarrow \quad (\alpha-\beta) \floor{\frac{\mu}{\alpha-\beta}} \leq \mu,
\end{align*}
which is obviously true. 
Therefore,
\begin{align*}
    \sum_{\bj \in \Delta_\mu} 2^{\delta\abs{\bj}_1} 
    \geq 2^{\delta\abs{\bj^*}_1} 
    = 2^{\delta{\floor{\mu/(\alpha-\beta)}}}
    \geq 2^{\delta \mu/(\alpha-\beta) - \delta} \sim 2^{\delta \mu/(\alpha-\beta)}.
\end{align*}

\emph{Step 2 (Upper bound). } 
If $d=1$, we have $j\in\Delta_\mu$ iff $j\leq \mu/(\alpha-\beta)$ and thus
$$
    \sum_{\bj \in \Delta_\mu} 2^{\delta \abs{\bj}_1} 
    \leq \sum_{j=0}^{\ceil{\mu/(\alpha-\beta)}} 2^{\delta j} 
    \sim 2^{\delta \ceil{\mu/(\alpha-\beta)}} 
    \leq 2^{\delta \big( \mu/(\alpha-\beta) + 1\big)} 
    \sim 2^{\delta \mu / (\alpha-\beta)}.
$$

Now let $d\geq 2$.
For each $\bj=(j_1,\ldots,j_d)\in\N_0^d$ set $\bj':=(j_2,\ldots,j_d)$.
Further, for $i=1,\ldots,d$ let $\mathfrak{J}_i:=\{\bj \in \N_0^d \sep j_i=\abs{\bj}_\infty\}$. 
Due to symmetry, it suffices to estimate
$$
    \sum_{\bj \in \Delta_\mu} 2^{\delta\abs{\bj}_1} 
    \leq \sum_{i=1}^d \sum_{\bj \in \mathfrak{J}_i\cap \Delta_\mu} 2^{\delta\abs{\bj}_1} 
    = d \sum_{\bj \in \mathfrak{J}_1\cap \Delta_\mu} 2^{\delta\abs{\bj}_1}.
$$
If $\bj=(j_1,\bj')\in \mathfrak{J}_1$, then $\alpha-\beta > 0$ yields
$$
    \bj \in \Delta_\mu
    \quad \Longleftrightarrow \quad \alpha (\abs{\bj'}_1+j_1) - \beta j_1 \leq \mu
    \quad \Longleftrightarrow \quad j_1 \leq \frac{\mu - \alpha \abs{\bj'}_1}{\alpha-\beta}.
$$
Therefore, $\bj\in \mathfrak{J}_1\cap \Delta_\mu$ implies $\abs{\bj'}_\infty \leq \abs{\bj}_\infty = j_1 \leq (\mu - \alpha \abs{\bj'}_1)/(\alpha-\beta)$ and hence
$$
    \alpha \abs{\bj'}_1 + (\alpha-\beta) \abs{\bj'}_\infty \leq \mu.
$$
Combining these estimates we conclude
\begin{align*}
    \sum_{\bj \in \Delta_\mu} 2^{\delta\abs{\bj}_1}
    \lesssim \sum_{\bj \in \mathfrak{J}_1\cap \Delta_\mu} 2^{\delta\abs{\bj}_1}
    \leq \sum_{\substack{\bj' \in \N_0^{d-1}:\\\alpha \abs{\bj'}_1 + (\alpha-\beta) \abs{\bj'}_\infty \leq \mu }} 2^{\delta\abs{\bj'}_1} \sum_{j_1=\abs{\bj'}_\infty}^{\ceil{(\mu - \alpha \abs{\bj'}_1)/(\alpha-\beta)}} 2^{\delta j_1}.
\end{align*}
Since $\delta>0$, up to constants the inner geometric sum is upper bounded by
$$
    2^{\delta \ceil{(\mu - \alpha \abs{\bj'}_1)/(\alpha-\beta)}} 
    \leq 2^{\delta \big( (\mu - \alpha \abs{\bj'}_1)/(\alpha-\beta) + 1 \big)} 
    \lesssim 2^{\delta (\mu - \alpha \abs{\bj'}_1)/(\alpha-\beta)}
$$
such that
\begin{align*}
    \sum_{\bj \in \Delta_\mu} 2^{\delta\abs{\bj}_1}
    &\lesssim \sum_{\substack{\bj' \in \N_0^{d-1}:\\\alpha \abs{\bj'}_1 + (\alpha-\beta) \abs{\bj'}_\infty \leq \mu }} 2^{\delta \big( \abs{\bj'}_1 + (\mu - \alpha \abs{\bj'}_1)/(\alpha-\beta)\big)} \\
    &\leq 2^{\delta \mu /(\alpha-\beta)} \sum_{\bj' \in \N_0^{d-1}} 2^{\delta \big( 1 - \alpha /(\alpha-\beta)\big)\abs{\bj'}_1} \\
    &= 2^{\delta \mu /(\alpha-\beta)} \prod_{i=1}^{d-1} \sum_{j_i=0}^{\infty} 2^{-\delta \beta/(\alpha-\beta) j_i} \\
    &\lesssim 2^{\delta \mu /(\alpha-\beta)},
\end{align*}
where we used that due to the assumption $\alpha>\beta>0$ we have $\delta \beta/(\alpha-\beta)>0$.
\end{proof}

Finally, our proofs of the upper bounds in \autoref{sect:upper} make use of
\begin{lemma}\label{lem:log-killer}
Let $d\in\N$, $\alpha > \beta \geq 0$, and $\delta>0$. Then
$$
    \sum_{\bj \in \mathfrak{L}_\mu} 2^{-\delta (\abs{\bj}_1-\abs{\bj}_\infty)} \lesssim 1, \qquad \mu\in\N\setminus\{1\}.
$$
\end{lemma}
\begin{proof}
If $d=1$, \autoref{lem:countingS} yields
$$
    \sum_{\bj \in \mathfrak{L}_\mu} 2^{-\delta (\abs{\bj}_1-\abs{\bj}_\infty)} = \sum_{\bj \in \mathfrak{L}_\mu} 1 = \abs{\mathfrak{L}_\mu} \sim 1, \qquad \mu\in\N\setminus\{1\}.
$$
So let $d\geq 2$ and define $\mathfrak{J}_i:=\left\{\bj=(j_1,\ldots,j_d) \in \N_0^d \sep j_i=\abs{\bj}_\infty\right\}$ for $i=1,\ldots,d$. Then
\begin{align*}
    \sum_{\bj \in \mathfrak{L}_\mu} 2^{-\delta (\abs{\bj}_1-\abs{\bj}_\infty)} 
    \leq \sum_{i=1}^d \sum_{\bj \in \mathfrak{J}_i \cap \mathfrak{L}_\mu} 2^{-\delta (\abs{\bj}_1-\abs{\bj}_\infty)}
    = d \sum_{\bj=(j_1,\bj') \in \mathfrak{J}_1 \cap \mathfrak{L}_\mu} 2^{-\delta \abs{\bj'}_1},
\end{align*}
where every $\bj=(j_1,\bj') \in \mathfrak{J}_1 \cap \mathfrak{L}_\mu$ satisfies $\mu-1 < \alpha (j_1 + \abs{\bj'}_1) - \beta j_1 \leq \mu$, i.e.
$$
    \frac{\mu-1}{\alpha-\beta} - \frac{\alpha}{\alpha-\beta} \abs{\bj'}_1
    < j_1 
    \leq \frac{\mu}{\alpha-\beta} - \frac{\alpha}{\alpha-\beta} \abs{\bj'}_1.
$$
Thus, independent of $\mu$ there are only constantly many different values for $j_1$ for fixed $\bj'$.
So,
\begin{align*}
    \sum_{\bj \in \mathfrak{L}_\mu} 2^{-\delta (\abs{\bj}_1-\abs{\bj}_\infty)} 
    &\lesssim \sum_{\bj=(j_1,\bj') \in \mathfrak{J}_1 \cap \mathfrak{L}_\mu} 2^{-\delta \abs{\bj'}_1} 
    \lesssim \sum_{\bj'\in\N_0^{d-1}} 2^{-\delta \abs{\bj'}_1} 
    = \prod_{i=1}^{d-1} \sum_{j_i=0}^{\infty} 2^{-\delta j_i} 
    \lesssim 1.\qedhere
\end{align*}
\end{proof}

\end{appendices}

\phantomsection
\addcontentsline{toc}{section}{References}
\bibliographystyle{is-abbrv}
\small
\bibliography{Bibliography.bib}

\end{document}